\documentclass[reqno,11pt]{amsart}

\usepackage{amsmath,amssymb}
\usepackage{amsthm}
\usepackage{hyperref}
\usepackage{mathrsfs}
\usepackage{palatino}
\usepackage{xspace}

\newtheorem{theorem}{Theorem}
\newtheorem{corollary}[theorem]{Corollary}
\newtheorem{proposition}[theorem]{Proposition}
\newtheorem{lemma}[theorem]{Lemma}

\newtheoremstyle{ourremark}
  {3pt}
  {3pt}
  {}
  {}
  {\bfseries}
  {}
  {.5em}
  {}
\theoremstyle{ourremark}
\newtheorem{remark}[theorem]{Remark}
\newtheorem{definition}[theorem]{Definition}

\newcommand{\no}{\noindent}
\newcommand{\Lie}{{\mathcal L}}
\newcommand{\conv}{\text{\rm conv}}
\newcommand{\sn}{\mathrm{sn}}
\newcommand{\ct}{\mathrm{ct}}
\newcommand{\BLOB}{\textsf{bLob}\xspace}

\newcommand{\xqed}{%
\leavevmode\unskip\penalty9999 \hbox{}\nobreak\hfill
\quad\hbox{\ensuremath{\qed}}}

\def\R{\mathbb{R}}

\def\dfn#1{{\em #1}}
\def\half{\frac{1}{2}}

\def\sec{\mathop{\mathrm{sec}}}
\def\II{I\!I}
\def\Lie{\mathcal{L}}
\newcommand\g[1]{\left\langle #1 \right\rangle}
\def\F{\mathscr{F}}  
\def\xistd{\xi_\text{std}} 
\def\C{\mathsf{C}}
\def\T{\mathsf{T}}
\def\D{\mathsf{D}}
\def\nD{n_\mathsf{D}}
\def\hdim{{2 n+1}}
\def\tauPS{\tau^\text{\rm PS}}

\newcommand{\rmax}{r_{\max}}
\newcommand{\rtransverse}{r_\pitchfork}
\newcommand{\rtamed}{r_\tau}

\newcommand{\injg}{\text{\rm inj}(g)}
\newcommand{\M}{K}
\newcommand{\bound}{\rho}
\newcommand{\Ric}{\mathrm{Ric}}

\DeclareMathOperator{\Rm}{Rm}
\newcommand{\B}{%
\frac{\theta'}{2} +%
\sqrt{ \frac{(\theta')^2}{4}%
-\frac{1}{2}\min_{u \in \xi,
	\|u\| = 1}  \Bigl(K(u, R_\alpha) + K(Ju, R_\alpha)\Bigr)}%
}
\newcommand{\A}{\frac{4}{3}\sqrt{2 n-1}\,|\text{\rm sec}(g)|}

\numberwithin{equation}{section} 
\numberwithin{theorem}{section}

\title[Quantitative Darboux theorems]{Quantitative Darboux theorems\\ in contact geometry}

\author{John B. Etnyre} 

\address{School of Mathematics,
Georgia Institute of Technology,
Atlanta, GA 30332} \email{etnyre@math.gatech.edu}
\urladdr{http://www.math.gatech.edu/\~{}etnyre/}

\author{Rafal Komendarczyk}

\address{Department of Mathematics,
Tulane University,
New Orleans, LA 70118}  \email{rako@tulane.edu}
\urladdr{http://www.math.tulane.edu/\~{}rako/}

\author{Patrick Massot}
\address{Universit\'e Paris Sud, 91405 Orsay Cedex, France}
\curraddr{Centre de Math\'ematiques Laurent Schwartz\\
  \'Ecole Polytechnique \\
  91128 Palaiseau Cedex \\
  FRANCE}
\email{patrick.massot@polytechnique.edu}
\urladdr{http://www.math.polytechnique.fr/perso/massot.patrick/}

\begin{document}
\begin{abstract}
This paper begins the study of relations between Riemannian geometry and contact
topology on $(2n+1)$--manifolds and continues this study on 3--manifolds.
Specifically we provide a lower bound for the radius of a geodesic ball in a
contact $(2n+1)$--manifold $(M,\xi)$ that can be embedded in the standard
contact structure on $\R^{2n+1}$, that is on the size of a Darboux ball. The
bound is established with respect to a Riemannian metric compatible with an
associated contact form $\alpha$ for $\xi$. In dimension three, this further 
leads us to an estimate of the size for a standard neighborhood 
of a closed Reeb orbit. The main tools are classical
comparison theorems in Riemannian geometry.  In the same context, we also use
holomorphic curves techniques to provide a lower bound for the radius of a
PS-tight ball. 
\end{abstract}

\maketitle

\vspace{-1cm}
\section{Introduction}

Darboux's theorem in contact geometry says that any point in a $(2n+1)$
dimensional contact manifold has a neighborhood that can be identified with an
open ball in $\R^{2n+1}$ with its standard contact structure. In
\cite{EKM-sphere} it was shown that a quantitative version of Darboux's theorem
can give interesting global information about a contact structure on a
3--manifold. For example one can give Riemannian geometric criteria for a
contact structure to be universally tight and in addition prove a contact
geometric version of the sphere theorem.  Such results rely on deep theorems
about contact 3--manifolds. Our understanding of contact manifold in higher
dimensions is much less advanced but we will still be able to prove a
quantitative version of Darboux's theorem. 

Given a contact structure $\xi$ on a $(2n+1)$--manifold $M$ and a
Riemannian metric $g$ we can define the \dfn{Darboux radius} of $(\xi, g)$ at a point $p\in M$ as 
\begin{align*}
\begin{split}
	\delta_p(\xi, g)& =\sup\{r < \mathrm{inj}(g, p) \,|\, \text{the open geodesic ball $(B_p(r),\xi)$
  at $p$ of radius $r$ is}\\ 
  &\qquad\qquad  \text{  contactomorphic to an open subset in $(\R^{2\,n+1},\xi_{\text{\rm std}})$}\},
\end{split}
\end{align*}
(where $\mathrm{inj}(g, p)$ is the injectivity radius of $g$ at $p$) and the
\dfn{Darboux radius of $(\xi, g)$} to be
\[
  \delta(\xi, g)=\inf_{p\in M} \delta_p(\xi, g). 
\]
One would like to estimate these quantities in terms of $g$. In dimension 3 this
was done in \cite{EKM-sphere} but relied heavily on a theorem of Eliashberg
\cite{Eliashberg92a} that says that any tight contact structure on a
3--ball is embeddable in the standard contact structure on $\R^3$. In particular
in \cite{EKM-sphere} we defined the \dfn{tightness radius at a point $p$}, $\tau_p$,
and the \dfn{tightness radius} $\tau$
for a contact metric 3-manifold. Eliashberg's theorem identifies $\tau_p$ and
$\tau$ with  $\delta_p$ and $\delta$, respectively. As Eliashberg's theorem is
unavailable to us in higher dimensions and it is even plausible that it does not
hold, we will discuss direct arguments to estimate $\delta$.

But first, a more direct generalization of the estimates proved in
\cite{EKM-sphere} would be to estimate the size of geodesic balls in contact
manifolds that were in some sense ``tight''.  
We say a contact structure $\xi$ on $M^{2n+1}$ is \dfn{PS-overtwisted} if it
contains a \BLOB (see Section~\ref{SS:PS} for the relevant definitions). 
Otherwise we say $\xi$ is \dfn{PS-tight}. Given a
contact $(2n+1)$--manifold $(M,\xi)$ and a Riemannian metric $g$ we  now define
the \dfn{PS--tightness radius} at $p\in M$ with respect to $g$ to be  
\begin{align*}
\tauPS_p(\xi, g) =\sup\{r < \mathrm{inj}(g, p)\,|&\, \text{the geodesic ball $B_p(r)$
  at $p$}\\
  &  \text{of radius $r$ is PS--tight}\},
\end{align*}
and the \dfn{PS-tightness radius of $M$} to be
\[
\tauPS(\xi, g)  =\inf_{p\in M} \tauPS_p(\xi, g).
\]

After this paper was written, Borman, Eliashberg and Murphy
\cite{BormanEliashbergMurphy} defined a notion of an ``overtwisted
disk'' in all dimensions and proved that contact structures containing
such a disk satisfy an $h$-principle. In particular, all almost contact structures are homotopic to contact structures with an overtwisted disk and any two such contact structures that are homotopic through almost contact structures are isotopic. So this notion of overtwisted
is ``the right one'' in higher dimensions and we can use
it to define the tightness radius $\tau(\xi, g)$. It is proved in
\cite{MNW, Niederkrueger06} that the standard structure on $\R^{2n+1}$ is
PS-tight and in 
\cite{BormanEliashbergMurphy} that an overtwisted contact structure is
PS-overtwisted so 
$\delta(\xi, g)\leq \tauPS(\xi, g)\leq \tau(\xi, g)$. The relation between
tightness and PS-tightness is not yet completely elucidated and they may be
equivalent. In any case, all our results about $\tauPS(\xi, g)$ (which are 
lower bounds) also hold equally well for $\tau(\xi,g)$. 

Again, in dimension $3$ the above quantities all are equal 
\[
\tauPS(\xi, g)= \tau(\xi, g)=\delta(\xi, g),
\]
because in that dimension {\it PS-tight} is equivalent
to {\it tight}, which, for a 3--ball, in turn is equivalent to being a Darboux
ball. In higher dimensions it is not known whether $\tauPS(\xi, g)$ can be
strictly larger than $\delta(\xi, g)$.

\subsection{Estimates for \texorpdfstring{$\tauPS$}{tauPS}} The first result in
this article extends convexity type estimates for $\tauPS(M,\xi)$ from
\cite{EKM-sphere} to higher dimensions in the setting of compatible metrics. In
higher dimensions the definition of a metric $g$ being compatible with a contact
structure $\xi$ on $M^{2\,n+1}$ is more complicated than the one considered in
\cite{EKM-sphere} for dimension 3. We refer to Definition~\ref{highDcompat} for
the precise details. Here we merely note that given a contact form $\alpha$ for
$\xi$ used in the definition of compatibility,
Proposition~\ref{proposition:determinemetric} gives a  complex structure
$J$ on $\xi,$ and the metric can be written as follows
\[
g(v,w) = \frac{1}{\theta'} d\alpha(v, Jw^\xi)+\alpha(v)\alpha(w),
\]
where $\theta'$ is constant and measures the instantaneous rotation speed of
$\xi$ (see Remark~\ref{rem:rotation_speed}), and $w^\xi$ denotes the
component of the vector $w$ lying in the $\xi$ component of the
splitting of $TM$ into $\xi$ and the span of the Reeb vector field of
$\alpha$. 

We obtain the following generalization of Theorem~1.3 from \cite{EKM-sphere}.

\begin{theorem}\label{thm:main-higher} 
Let $(M^\hdim,\xi)$ be a contact manifold and $(\alpha, g,  J)$ 
be a compatible metric structure for $\xi.$ Then,
\begin{equation}\label{eq:conv}
\tauPS(\xi, g)\geq \conv(g),
\end{equation}
where 
\[
\begin{split}
\conv(g)=\sup \{ r \, |\,& r< \injg \text{ and the geodesic spheres of radius $r$} \\ 
&\text{are geodesically convex} \},
\end{split}
\]
and $\injg$ is the injectivity radius of $(M,g).$ 
In particular, if $\sec(g)\leq K$, for $K > 0$,
then
\[
 \tauPS(\xi, g)\geq \min\{\injg,\frac{\pi}{2\sqrt{K}}\}
\]
and $\tauPS(\xi, g) = \injg,$ if $g$ has non-positive curvature.
\end{theorem}
As in dimension 3 the bound on $\tauPS$ is especially effective in the case of
non-positive curvature. 

\begin{corollary}\label{cor:ut-higher}
 Let $(M,\xi)$ be a $(2n+1)$-contact manifold and $g$ a complete Riemannian
 metric compatible with $\xi$ having non-positive sectional curvature.
 Then $\xi$ pulled back to the universal cover of $M$ is PS-tight. In
 particular, if $(M,\xi)$ contains a \BLOB $N$, then the image of $i_*:\pi_1(N)\to\pi_1(M)$, 
 where $i:N\to M$ is the inclusion map, is infinite.  
\end{corollary}

\begin{remark}
Sasakian (or, more generally, $K$-contact) manifolds are contact metric manifolds
satisfying some extra conditions, see \cite{Blair02}. One can naturally wonder whether 
Theorem \ref{thm:main-higher} is relevant to their study. But those manifolds do not contain any
\BLOB since one can combine the main results of \cite{NiederkrugerPasquotto09} and
\cite{Rukimbira95} to prove that they are symplectically fillable. 
The same conclusion actually holds for the wider class of integrable
$CR$-contact metric manifolds in dimension at least 5, see
Definition~\ref{integrableCR}. As explained in \cite[Theorem
5.60]{CieliebakEliashberg} this can be proved using deep results in complex
analysis due to Lempert, Hironaka and Rossi. However, 
Theorem~\ref{thm:GeometricDarbouxHigherRough} below still gives non-trivial 
information in this context.
\xqed
\end{remark}

\begin{remark}
Recall that in dimension 3 a contact structure is called universally
tight if the contact structure is tight when pulled back to the
universal cover and virtually overtwisted if it is overtwisted when
pulled back to a finite cover. Since an overtwisted disk is contractible
it can be lifted to any cover, thus if a contact structure is
universally tight then it is also tight. Since the fundamental groups of
3--manifolds are residually finite it is also known that being universally
tight is equivalent to not being virtually overtwisted. One can make the
same definitions in higher dimensions for PS-universally tight and
PS-virtually overtwisted and similarly for universally tight and
virtually overtwisted using the Borman-Eliashberg-Murphy notion of
overtwisted. Since their notion of overtwisted also involves disks one
again sees that universally tight implies not virtually overtwisted
which in turn implies tight, but now we do not know if  universally
tight is equivalent to not virtually overtwisted (though it seems
unlikely due to the fact that fundamental groups in higher dimensions
need not be residually finite). Note in particular that
Corollary~\ref{cor:ut-higher} implies that $(M,\xi)$ is universally
tight.  The relation between  PS-universally tight and PS-virtually
overtwisted and PS-tight  is much more complicated due to the fact that
a \BLOB need not be simply connected and hence it might not lift to
covers. Notice that the condition in Corollary~\ref{cor:ut-higher} on
the fundamental group of a \BLOB implies that its preimage in the
universal cover of $M$ will be non-compact and hence not a \BLOB. 
\xqed
\end{remark}

\subsection{Direct geometric methods for estimating the Darboux radius}\label{S:direct-intro}
We now discuss a method for estimating $\delta(\xi, g)$ in higher
dimensions. This strategy is more geometric and direct than the one used in
\cite{EKM-sphere} to bound the Darboux radius in dimension 3, as it does not use
holomorphic curves or classification results. We want a control on the Darboux
radius using control on curvature, the rotation speed $\theta'$ and the
tensor $[J, J]$ (the later measures how far is the induced CR structure from
being integrable, see Section~\ref{S:compatible_metrics} for
definitions). Note that, if $\M$ is a bound on sectional curvature then
$1/\theta'$, $\|[J, J]\|$ and $1/\sqrt \M$ behave like lengths under
homothety (constant rescaling of the metric). This explains the
appearance of such terms in the following estimate.  It is also expected
that the estimate deteriorates when the rotation speed increases or when
one widens the sectional curvature interval but also when $\|[J, J]\|$
increases since compatible metrics are less tied to the contact
structure for non-integrable CR structures.

\begin{theorem}
	\label{thm:GeometricDarbouxHigherRough} 
Let $(M^\hdim,\xi)$ be a  $(2n + 1)$--dimensional contact manifold and 
$(\alpha, g,  J)$ be a complete compatible metric structure for $\xi$ with
rotation speed $\theta'$. If the sectional curvature of $g$ is contained in the
interval $[-K,K]$ for some positive $K$ then 
\[
\delta(\xi,g) \geq 
\min\left(\frac{\injg}{2}, 
\frac{1}{208n^2\max\big(\sqrt K,\, \|[J, J]\|,\, \theta'\big)}\right),
\]
where $[J,J]$ is the Nijenhuis torsion of the complex structure $J$ on $\xi$.
\end{theorem} 
Theorem~\ref{thm:GeometricDarbouxHigher} is a better, but more technical, 
bound on $\delta(\xi,g) $ and
is proven in Section~\ref{sec:darboux} after the explicit quantities used in the
better estimate are explained, but the main point is the existence of an {\em explicit} bound
rather than its precise expression. 
\begin{remark}
Since every contact structure has a
compatible metric structure, this theorem always produces an explicit bound on the 
Darboux radius in terms of Riemannian curvature information.
\end{remark}

In dimension 3 we can improve the bound on
$\delta(\xi,g)$ coming from these geometric methods.

\begin{theorem}
\label{thm:geometric}
Let $(M, \xi)$ be a contact 3--manifold and $g$ a complete Riemannian metric that is compatible with $\xi$ and has rotation speed
$\theta'$. Then,
\begin{equation}\label{eqn:main-bound}
\delta(\xi, g)= \tau(M,\xi)\geq \min\Bigl\{
\frac{\injg}{2}, \frac{\pi}{2\sqrt{K}},\frac{2}{\sqrt{2\,A+B^2}+B}\Bigr\},
\end{equation}
where
\[
A = \frac{4}{3}\,|\sec(g)|,\quad
B = \frac{\theta'}{2} + 
	 \sqrt{ \frac{(\theta')^2}{4}-\frac{1}{2}\min_{p\in M} \bigl(\Ric_p(R_\alpha)\bigr)}
\]
and $\Ric$ is the Ricci tensor, $R_\alpha$ is the Reeb vector field of $\alpha$ and  
$|\text{\rm sec}(g)|$ is the maximum in absolute value of the sectional curvature
over $(M,g)$, and the constant $K>0$ an arbitrary positive upper bound for it. 
\end{theorem}

Theorems~\ref{thm:GeometricDarbouxHigherRough} and~\ref{thm:geometric}
provide bounds on $\delta(\xi, g)$ which are weaker than the one given
in Theorem~\ref{thm:main-higher} in dimension 3. But the ideas
underlying their proof extend to higher dimensions and can be used when
both the convexity of the boundary assumption of
Theorem~\ref{thm:main-higher} and the absence of closed Reeb orbits
fail, thus ruining the strategy used to prove
Theorem~\ref{thm:main-higher}. As an example of such a situation we show
the following estimate on the size of a standard neighborhood of a
closed Reeb orbit in a contact 3--manifold. 

\begin{theorem}
\label{thm:Reeb_nbhds}
Let $(M, \xi)$ be a contact 3--manifold and $g$ a complete Riemannian metric
that is compatible with $\xi.$ Let $\gamma$ be a closed Reeb orbit and $\T(r)$
an embeded  geodesic tube of radius $r$ about $\gamma.$ 
If $r$ is below the bound of Equation \eqref{eqn:main-bound} then the contact
structure restricted to $\T(r)$ is universally tight and, moreover, can be
embedded in $(S^1\times \R^2, \xistd=\ker (d\phi + r^2\, d\theta)).$
\end{theorem}

\subsection{Outline}

In Section~\ref{S:compatible_metrics} we define the notions of compatibility
between metrics and contact structures in an arbitrary dimension. 
The paper then splits into two logically independent threads which both depend
on Section~\ref{S:compatible_metrics}.  The first one is covered in Section~\ref{sec:tauPSestimate} which
compares Riemannian and almost complex convexity and proves Theorem~\ref{thm:main-higher} and Corollary~\ref{cor:ut-higher}. 
The second one starts in Section~\ref{sec:darboux} which
states the refined version of Theorem~\ref{thm:GeometricDarbouxHigherRough} and
proves it modulo a number of propositions which are proved in
Subsections~\ref{S:t_p-est} to~\ref{sec:cyl2ball}.
Section~\ref{S:geometric_method} explains how geometrical methods of
Section~\ref{sec:darboux} can be strengthened using topological methods which
are specific to dimension 3 and proves Theorems~\ref{thm:geometric}
and~\ref{thm:Reeb_nbhds}.

\subsection*{Acknowledgments}  The authors wish to thank Vladimir Krouglov for
useful correspondence and Yasha Eliashberg for pointing out the fillability of
integrable CR manifolds. 
We also thank the referee for pointing out several potential pitfalls and 
providing highly valuable input that improved the paper.
The second author is grateful to Chris Croke for
enjoyable lunch meetings, and many helpful suggestions to the project. The first
author was partially supported by the NSF Grant DMS-0804820 and DMS-1309073. 
The second author was partially supported by DARPA YFA-N66001-11-1-4132.  The 
third author was partially supported by the ANR grant ANR-10-JCJC 0102.

\section{Metrics compatible with contact structures}
\label{S:compatible_metrics}

Throughout this section and the rest of the paper we will interchangeably use
the notation $g(u,v)$ and $\g{u,v}$ to denote the Riemannian metric evaluated on
the vectors $u$ and $v.$

On a vector space $V$, let $\omega$ be a symplectic pairing, $g$  an
inner product and $J$ a complex structure. Recall that $(g, \omega, J)$ is
called a compatible triple if $g(\cdot, \cdot) = \omega(\cdot, J\cdot)$. In particular $J$ is
tamed by $\omega$, that is $\omega(v,Jv)>0$ for $v\not =0$, and we also
have $\omega(Ju,Jv)=\omega(u,v)$.
Note that two members of the triple uniquely determine the third one. We say
that $g$ and $\omega$ are compatible if there exists some (unique) $J$ such
that $(g, \omega, J)$ is a compatible triple.

In this paper we say that $g$ and $\omega$ are \dfn{weakly compatible} if there
is a positive number $c$ such that $cg$ and $\omega$ are compatible. In that
case, volume forms induced by $cg$ and $\omega$ on $V$ are equal hence $c$ and
then $J$ are uniquely determined by $\omega$ and $g$.

\begin{definition}
We say a contact structure $\xi$ and a Riemannian metric $g$ on a
$(2n+1)$--manifold $M$ are \dfn{weakly compatible} if there is a contact form
$\alpha$ for $\xi$ such that its Reeb vector field $R_\alpha$ is orthogonal to
$\xi$ and $g|_\xi$ and $(d\alpha)|_\xi$ are, pointwise, weakly compatible.
\end{definition}

We note right away that the contact form involved in the above definition is
unique up to a constant multiple because of the following well-known
lemma.

\begin{lemma}
Two contact forms on a connected manifold which have the same kernel and
parallel Reeb fields are constant multiple of each other.
\end{lemma}

\begin{proof}
Given two such forms $\alpha$ and $\alpha'$, the kernel condition means
there is some positive function $f$ such that $\alpha' = f\alpha$.
By Darboux's theorem for contact \emph{forms}, each point lies in a 
coordinates chart where $\alpha = dz - \sum y_i dx_i$ and its Reeb field is
$\partial_z$. We can compute:
\[
	\iota_{\partial_z}d\alpha' = 
	-\sum(y_i\partial_z f + \partial_{x_i}f)dx_i
	-\sum \partial_{y_i}fdy_i.
\]
So the Reeb field condition is equivalent to
$\partial_{y_i} f = 0$ and $y_i\partial_z f + \partial_{x_i}f = 0$ for all $i$.
Differentiating the second equation with respect to $y_i$ one sees that 
$\partial_z f=0$ and then that $\partial_{x_i}f = 0$ for all $i$. Thus $df = 0$ 
and $f$ is constant.
\end{proof}

One then has the following mostly tautological proposition.
\begin{proposition}
\label{proposition:determinemetric}
Let $\xi$ be a contact structure on a (2n+1)--manifold $M$ and $g$ a  weakly
compatible metric. Then fixing a contact form $\alpha$ as in the
definition of weakly compatible we denote by $c$ the positive function and $J$
the complex structure on $\xi$ such that $(cg|_\xi, d\alpha|_\xi, J)$ is a
compatible triple on $\xi$.

The  complex structure $J$ can be extended to a linear map on $TM$ as
follows
\begin{equation}\label{definephi}
\phi: TM\to TM: v\to Jv^\xi,
\end{equation}
where 
\begin{equation}\label{eq:xi-proj}
v^\xi=v-g(v,n)n
\end{equation}
is the component of $v$ lying in $\xi$ (here
$n=R_\alpha/\|R_\alpha\|$ is the unit normal vector to $\xi$). The metric $g$ may
then be expressed as
\begin{equation}\label{definegHiD}
g(u,v)=\frac{\rho}{\theta'} d\alpha(u,\phi(v)) + \rho^2 \alpha(u)\alpha(v),
\end{equation}
where $\rho=\| R_\alpha\|$ and $\theta'=\rho\, c$.
\xqed
\end{proposition}
\begin{remark}\label{rem:rotation_speed}
Our choice of notation $\theta'$ in the above proposition can be
explained as follows. Let $\xi$ be a cooriented hyperplane field on any
Riemannian manifold $(M, g)$. Let $n$ be the positive unit normal
vector field of $\xi$ and let $\beta = g(n, \cdot)$ be the unit 1-form
defining $\xi$. The restriction of $d\beta$ to $\xi$ measures how fast
$\xi$ rotates along vector fields tangent to $\xi$. Indeed if $X$ and
$Y$ are two such vector fields with unit norm and $\varphi_t$ is the flow of
$X$ then the derivative at $t=0$ of the angle between $n$ and $\varphi_t^*Y$ is
$d\beta(X, Y)$ (this results from a short computation). If $\xi$ is a contact
structure and $g$ is weakly compatible with $\xi$ then
Equation~\eqref{definegHiD} and a short computation prove that $\theta' =
d\beta(u, Ju)$ for any unit vector $u$ in $\xi$. Hence $\theta'$ is indeed the
derivative of some angle.
\end{remark}

\begin{remark}\label{rem:phi2}
It is useful to notice that $\phi^2$ can be expressed as
\begin{equation}
\label{square}
\phi^2(v)=-v+\alpha(v)R_\alpha,
\end{equation}
for all vectors $v\in TM.$
\end{remark}

It is difficult to say much about weakly compatible metrics in higher
dimensions, but if we assume the length of $R_\alpha$ is constant then we can
understand something about the covariant derivatives of vectors with respect
to $R_\alpha$ and observe that the flow of the Reeb vector field traces out
geodesics. 
\begin{proposition}
\label{Rderivative}
Let $g$ be a metric weakly compatible with the contact structure $\xi$ on the
$(2n+1)$--manifold $M.$ Let $R_\alpha$ be the Reeb vector field associated to
the form $\alpha$ implicated in the definition of $g$ being weakly compatible
with $\xi.$ If $\rho=\|R_\alpha\|$ is constant, then 
\begin{equation}
\label{eqn:nabla_NN}
\nabla_{R_\alpha} R_\alpha = 0.
\end{equation}
In particular, flow lines of $R_\alpha$ are geodesics. Moreover, if $v$ is a
vector field tangent to $\xi$ then $\nabla_{R_\alpha}v$ is also tangent to
$\xi.$
\end{proposition}

\begin{proof}
We first notice that $\alpha$ and $R_\alpha$  are related by
\begin{equation}
\label{alphaRalpha}
\rho^2\alpha= \iota_{R_\alpha}g,
\end{equation}
since they agree on $\xi$ and on $R_\alpha.$
For any vector field $v$ the definition of the Reeb vector field implies
$d\alpha(R_\alpha, v) = 0$. On the other hand, using the assumption that $\rho$
is (a non-zero) constant, we compute
\begin{align*}
d\alpha(R_\alpha, v) 
&= R_\alpha\cdot \alpha(v)-v\cdot \alpha(R_\alpha)-\alpha([R_\alpha, v])\\
& = R_\alpha\cdot\alpha(v)-\alpha(\nabla_{R_\alpha} v-\nabla_v R_\alpha)\\
&=\rho^{-2} (R_\alpha\cdot \g{R_\alpha, v} -
\g{R_\alpha, \nabla_{R_\alpha} v-\nabla_v R_\alpha})\\
&=\rho^{-2} (\g{\nabla_{R_\alpha}R_\alpha, v} +
\g{R_\alpha, \nabla_{R_\alpha} v}-\g{R_\alpha, \nabla_{R_\alpha} v} +
\g{R_\alpha, \nabla_v R_\alpha})\\
&=\rho^{-2} (\g{\nabla_{R_\alpha}R_\alpha, v}),
\end{align*}
where the last equality follows since 
$2\g{R_\alpha, \nabla_v R_\alpha} = v\cdot \g{R_\alpha,R_\alpha}=0.$
So $\g{\nabla_{R_\alpha} R_\alpha, v} = 0$ for every $v$ and thus 
$\nabla_{R_\alpha} R_\alpha = 0.$

If $v$ is a vector field tangent to $\xi$ then 
\begin{align*}
\g{R_\alpha,\nabla_{R_\alpha}v} = \g{\nabla_{R_\alpha}R_\alpha, v} +
\g{R_\alpha, \nabla_{R_\alpha}v} = R_\alpha\cdot\g{R_\alpha, v} = 0,
\end{align*}
so $\nabla_{R_\alpha} v$ is tangent to $\xi.$
\end{proof}

We will not say more about weak compatibility in higher dimensions and
restrict our attention to the stronger notion of compatibility. 
\begin{definition}
\label{highDcompat}
We say a contact structure $\xi$ and a Riemannian metric $g$ on a
$(2n+1)$--manifold $M$ are \dfn{compatible} if there is a contact form $\alpha$
for $\xi$ such that its Reeb vector field $R_\alpha$ is orthogonal to $\xi,$ has
unit length $\| R_\alpha\|=1,$ and $g|_\xi$ and $(d\alpha)|_\xi$ are, pointwise,
weakly compatible and the function $\theta',$ defined in
Proposition~\ref{proposition:determinemetric}, is constant.

Equivalently, we can say $g$ and $\xi$ are \dfn{compatible} if there is a
contact form $\alpha$ for $\xi$ and a  complex structure $J$ on $\xi$
that is compatible with $d\alpha$ on $\xi,$ satisfies
\begin{equation}\label{eq:g-alpha}
g(u,v)=\frac 1{\theta'} d\alpha(u, Jv),
\end{equation}
for all $u,v\in \xi,$ where $\theta'$ is some positive constant and for which
the Reeb vector field $R_\alpha$ is the unit normal to $\xi.$ 
\end{definition}

Recall \cite[p. 63]{Blair02} that the {\it Nijenhuis torsion} $[T,T]$ of a
$(1,1)$--tensor field $T$ is a skew--symmetric tensor field of type $(1,2)$
defined as 
\begin{equation}
 [T,T](X,Y)=T^2[X,Y]+[TX,TY]-T[TX,Y]-T[X,TY].
\end{equation}

\begin{definition}
\label{integrableCR}
Let $\xi$ be a (cooriented) contact structure on a $(2n+1)$--manifold $M$ and $\alpha$ a
contact form for $\xi$. A complex structure $J$ on $\xi$ is called a
\dfn{$CR$--structure} (or more specifically a \dfn{strictly pseudo-convex
integrable $CR$--structure}) if and only if $J$ is tamed by $d\alpha$ and the
Nijenhuis torsion $[J,J]$ of $J$ on $\xi$ vanishes. The later explicitly says that
for all sections $v$ and $w$ of $\xi$
\begin{equation}\label{eqn:integrability}
	[J,J](v,w)=-[v,w]+[Jv, Jw]- J([Jv, w] + [v, Jw])=0,
\end{equation}
\end{definition}

\no Notice that $[J,J]$ is well defined as a $(1,2)$--tensor field on $\xi$ because 
\[
\alpha([Jv, w] + [v, Jw]) = -d\alpha(Jv, w)-d\alpha(v,Jw)
= -d\alpha(Jv,w)+d\alpha(Jv,w) = 0.
\]
thus 
\begin{equation}
	\label{eqn:integrability2}
	[Jv,w]+[v,Jw] \text{ is tangent to } \xi.
\end{equation}
Similarly one can see that $[Jv, Jw]-[v,w]$ is tangent to $\xi.$

After \cite{Blair02}, we call $(M,\alpha, g, J)$ a \dfn{contact metric structure} whenever $\alpha$ and $J$ define the compatible metric $g$ as above, equivalently we say that $(\alpha,g,J)$ is a compatible metric structure on $(M,\xi)$. Additionally, if $J$
and $\xi$ define a $CR$-structure we say that $(M,\alpha, g, J)$ is a
\dfn{$CR$--contact metric structure}.

Note that the above integrability condition is automatic for any  complex structure
on a plane field in dimension $3$. Indeed, we may choose a (local) basis
$\{v,\phi(v)\}$ of $\xi$  and observe that
\begin{equation}\label{eqn:N-phi-3d}
	[J,J](v,v)=0,\quad \text{and} \quad [J,J](v,Jv)=0,
\end{equation}
from which Equation~\eqref{eqn:integrability} follows.

We are now ready to state relations between the various operators and their
derivatives. Most of these formulas are well known in the literature, {\em cf.\ }\cite{Blair02}, though extra
terms occur due to the generality we are considering here. We also note that
there are some sign discrepancies with \cite{Blair02}, coming from the fact that we are always using positive contact
structures. Due to this, and for the convenience of the reader, we provide
proofs of these formulas here.

Recall that  the \dfn{second fundamental form} $\II$ of $\xi$ is the quadratic form on $\xi$
defined as follows  \cite{Reinhart}: for vectors $u$ and $v$ in $\xi_p=T_pM\cap \xi,$
\begin{equation}\label{definitionII}
\II(u, v) = \frac{1}{2}\g{\nabla_u v+\nabla_v u, R_\alpha}.
\end{equation}
We also define the endomorphism 
\begin{equation}
h = \half \Lie_{R_\alpha} \phi:TM\to TM,
\end{equation}
where $\Lie$ denotes the Lie derivative.
This endomorphism is a repackaging of $\II$ (computations from the proof of
Proposition~\ref{CRformula} show that $h(R_\alpha) = 0$ and, for $u$, $v$ in
$\xi$, $\II(u, v) = \g{u, \phi(hv)}$). 

\begin{lemma}\label{rangeh}
Let $(\alpha, g, J)$ be a compatible metric structure on  $(M,\xi).$ Then $h(TM)\subset \xi.$
\end{lemma}
\begin{proof}
First compute
\begin{equation}
\label{eq:formulaforh}
h(v) = \frac 12 (\Lie_{R_\alpha}\phi)(v) 
= \frac 12 (\Lie_{R_\alpha}(\phi(v))-\phi(\Lie_{R_\alpha}v))
= \frac 12 ([R_\alpha, \phi(v)]-\phi([R_\alpha, v])).
\end{equation}
Since for any section $u$ of $\xi$ we have 
$\alpha([R_\alpha, u])=-d\alpha(R_\alpha, u)=0,$ we see that $[R_\alpha, u]$ 
is in $\xi$ for any vector field $u$ in $\xi.$ The result follows since the
image of $\phi$ is contained in $\xi.$  
\end{proof}

The basic relations between metric and contact geometric quantities are derived 
in the following proposition. 

\begin{proposition}
\label{CRformula}
Let $(\alpha, g, J)$ be a compatible metric structure on the contact
manifold $(M,\xi).$ Let $R_\alpha$ be the Reeb vector field of $\alpha$, $h$ the
endomorphism defined above and $\II$ the second fundamental form of $\xi$. 
Then the following equations hold
\begin{align}
\label{eqn:symgh}
 & \g{h(u),v}  =\g{u,h(v)},\\
\label{eqn:nabla_vN}
 & \nabla_v R_\alpha  = \phi \left(({\theta'}/{2})\,  v - h(v)\right),\displaybreak[0] \\
 \label{eqn:JhhJ}
 & (\phi h + h \phi)(v)  = 0,\displaybreak[0] \\
 \label{eqn:trII}
 & \II(v,v) + \II(J v, J v)  = 0, \displaybreak[0] \\
\intertext{and}
 \label{eqn:nabla_NJ}
 & \nabla_{R_\alpha} (Jv)  = J(\nabla_{R_\alpha} v),
\end{align}
for any $v,u\in \xi$ (in the last equation, $v$ has to be a vector field
tangent to $\xi$).
\end{proposition}

\no We will also need the following relations between curvature and contact geometric properties.

\stepcounter{equation}

\begin{proposition}
\label{3Dcurvature}
Let $(\alpha, g, J)$ be a compatible metric structure on  $(M,\xi)$, then
for any vectors $u$, $v$ and $w$
\begin{equation}
	\label{eq:nabla-reeb}
		\|\nabla R_\alpha\| = B
\end{equation}
and
\begin{equation}
	\label{eq:nabla-phi-norm}
	|\g{(\nabla_u \phi)(v), w}| \leq 
	\left(\half\|[J, J]\| + 2	B\right)\|u\|\cdot\|v\|\cdot\|w\|
\end{equation}
where 
\[
B = \B
\]
and $K(u,v)$ is the sectional curvature of the plane spanned by the unit vectors $u$ and $v$. 
\end{proposition}

\no A key to proving the above results is the computation of the covariant derivative
of $\phi.$
\begin{lemma}\label{lem:phiderivative}
Let $(\alpha, g, J)$ be a compatible metric structure on  $(M,\xi).$ Then for any vectors $u,v$ and $w$ the following equation
holds: 
\begin{equation}
\label{eqn:nabla-phi}
\begin{split}
\g{(\nabla_u \phi) (v), w} &= 
\half	\g{u^\xi, \; [J, J]\left(v^\xi, Jw^\xi\right) +
		\left(\left(2h - \theta' Id\right) \wedge \alpha\right) (v, w)}
\end{split}
\end{equation}
where, for any endomorphism field $A$, we set 
$\left(A \wedge \alpha\right)(v, w) = \alpha(w)A(v) - \alpha(v)A(w)$.
\end{lemma}

\begin{proof}[Proof of Proposition~\ref{CRformula}]
Throughout this proof we will repeatedly use the formula for $\phi^2$ given in
Equation~\eqref{square} without further notice.

We specialize Equation~\eqref{eqn:nabla-phi} to $u = R_\alpha$ to get that
$\nabla_{R_\alpha} \phi= 0$. Thus using
Equation~\eqref{eq:formulaforh} we see that
\begin{align*}
\g{2\, h(v), w} &= \g{[R_\alpha, \phi(v)] - \phi([R_\alpha, v]), w} =
 \g{(\nabla_{R_\alpha}\phi)(v)-\nabla_{\phi(v)} R_\alpha+\phi(\nabla_vR_\alpha),w}\\
&= \g{\phi(\nabla_vR_\alpha)-\nabla_{\phi(v)} R_\alpha,w}=
\g{R_\alpha, \nabla_{\phi(v)} w+\nabla_v \phi(w)}=\alpha(\nabla_{\phi(v)} w+\nabla_v \phi(w)).
\end{align*}
Note that since $d\alpha(\phi(v),w)=-d\alpha(v,\phi(w))$, we have $\alpha([\phi(v),w])+\alpha([v,\phi(w)])=0$ and therefore
\begin{align*}
\g{2\, h(v), w} &=\alpha(\nabla_{w} \phi(v)+\nabla_{\phi(w)} v)=\g{v, 2\, h(w)}.
\end{align*}
This proves
Equation~\eqref{eqn:symgh}.

For Equation~\eqref{eqn:nabla_vN} we specialize Equation~\eqref{eqn:nabla-phi}
to $v = R_\alpha$ and obtain 
\[
\begin{split}
\g{(\nabla_u \phi)	(R_\alpha), w} & =  
\half	\g{u^\xi, \; 	\left(2h - \theta' Id\right) \wedge \alpha(R_\alpha, w)} \\
 & =-\g{u^\xi, h(w)} + \theta'/2\g{u^\xi, w^\xi}\\
 & =-\g{h(u),w}+ \theta'/2\g{u^\xi, w^\xi}
\end{split}
\]
where the we used the symmetry Equation~\eqref{eqn:symgh} of $h$ in the last
equality. 
Since $\phi(R_\alpha)=0$, the left hand-side of the above is 
$\g{\nabla_u (\phi(R_\alpha)) - \phi(\nabla_u R_\alpha), w} = -\g{\phi
  (\nabla_u R_\alpha), w}.$ 
 Thus $\phi(\nabla_u R_\alpha)=h(u)-(\theta'/2) u^\xi.$ 
Applying $\phi,$ noting that $\nabla_u R_\alpha$ is tangent to $\xi$ and 
recalling that $\phi=J$ on $\xi,$ we establish Equation~\eqref{eqn:nabla_vN}.

We now prove that $h$ anti-commutes with $\phi,$ that is we prove
Equation~\eqref{eqn:JhhJ}
\begin{align*}
-\theta'\g{\phi(v), u}&+\g{u, h(\phi(v))+\phi(h(v))}
=\g{\frac{\theta'}{2}\phi(u) - \phi(h(u)), v} -	
\g{\frac{\theta'}{2}\phi(v) - \phi(h(v)), u}\\
&=\g{\nabla_u R_\alpha, v} - \g{\nabla_v R_\alpha, u}
=-\g{R_\alpha, [u,v]}=d\alpha(u,v)=-\theta'\g{u,\phi(v)}.
\end{align*}
The second equality follows from Equation~\eqref{eqn:nabla_vN} and the last
equality follows from the fact that for all $u$ and $v$ in $TM$ we have
\begin{equation}
\label{galpha} 
g(u,\phi(v))=-\frac 1{\theta'} d\alpha(u,v),
\end{equation}
which in turn follows from Equation~\eqref{definegHiD} and the fact that
$\phi(v)\in \xi.$ Continuing we see $\langle u, h(\phi(v))$ $+\phi(h(v))\rangle=0$ for all
$u\in \xi,$ thus establishing Equation~\eqref{eqn:JhhJ}.

To prove Equation \eqref{eqn:trII} (which is the only one which appears to be
new), we compute for any $v\in \xi$ (i.e. vector field extension of $v$)
\begin{align*}
\II(v,v) + \II(Jv, Jv) &= \g{\nabla_v v, R_\alpha} + \g{\nabla_{J v} J v,R_\alpha} 
= -\g{v, \nabla_v R_\alpha} - \g{J v, \nabla_{J v} R_\alpha}  \\
&= \g{v, -\theta'/2 J v - h (J v)} + \g{J v,
-\theta'/2 J^2 v - h (J^2 v)} \\
&= -\g{v, h (J v)} + \g{J v,  h (v)} =-\g{v, h (J v)} + \g{v,  h (J v)} =0,
\end{align*}
where the third equality follows from Equation~\eqref{eqn:nabla_vN} and
Equation~\eqref{eqn:JhhJ} has been used repeatedly.

We now prove Equation \eqref{eqn:nabla_NJ}. Let $v$ be a vector field tangent
to $\xi$. We use that $Jv = \phi(v)$ and $\nabla_{R_\alpha} \phi = 0$ to get
\[
\nabla_{R_\alpha}(Jv) = \nabla_{R_\alpha}(\phi(v)) = \phi(\nabla_{R_\alpha}v)
\]
In addition $\nabla_{R_\alpha}v$ is in $\xi$ by Proposition~\ref{Rderivative}. 
\end{proof}

\begin{proof}[Proof of Lemma~\ref{lem:phiderivative}]
To establish Equation~\eqref{eqn:nabla-phi} we notice that both sides of the
equation are tensors in $u,v$ and $w,$ so it suffices to establish the result
when $u,v$ and $w$ are chosen to be elements of a basis for $TM.$ We choose
(local)
vector fields $v_1,\ldots, v_n$ in $\xi$ such that $v_1, Jv_1,\ldots, v_n,Jv_n$ is an
oriented orthonormal basis for $\xi.$ In the computation below we assume
that $u,v$ and $w$ are chosen from the set $\{v_1, Jv_1,\ldots, v_n,
Jv_n, R_\alpha\}.$ Notice that this implies that the lengths of $u,v, w,
\phi(u),\phi(v)$ and $\phi(w)$ are constant as are their inner products with
each other.

Equation~\eqref{galpha} says that $\theta' g(u,\phi(v))=-d\alpha(u,v)$ for any
vectors $u$ and $v.$ Using this and the fact that $d\alpha$ is closed we have
\begin{equation}
\label{eqn:Phi-closed}
\g{[u, v], \phi(w)} + \g{[w, u], \phi(v)} + \g{[v, w], \phi(u)} = 0.
\end{equation}

Recall, the Koszul formula states that for any vector fields $u,v$ and $w$ 
\[
2\g{\nabla_u v, w}=u\cdot \g{v,w}+v\cdot \g{u,w} -w\cdot\g{u,v}+ 
\g{[u,v],w}+\g{[w,u],v}+\g{[w,v],u}.
\]
Using this, Equation~\eqref{square} to compute $\phi^2$ and the previous
equation we can begin our computation of $\nabla \phi$ as follows
\begin{align*}
2\g{(\nabla_u \phi)	(v), w} &=  2\g{\nabla_u (\phi(v)) - \phi(\nabla_u v), w} =
2\g{\nabla_u (\phi(v)), w} + 2\g{\nabla_u v- \alpha(\nabla_u v) 
R_\alpha, \phi(w)} \\
&= \g{[u, \phi(v)], w} + \g{[w, u], \phi(v)} + \g{[w, \phi(v)], u} \\
&\qquad\quad + \g{[u, v], \phi(w)} + \g{[\phi(w), u], v} + \g{[\phi(w), v], u} \\
&=
\g{[u, \phi(v)], w} + \g{[w, \phi(v)], u}\\
&\qquad\quad + \g{[\phi(w), u], v} + \g{[\phi(w), v], u} - \g{[v, w], \phi(u)}. 
\end{align*}
Substituting $\phi(v)$ for $v$ and $\phi(w)$ for $w$ in
Equation~\eqref{eqn:Phi-closed} and using Equation~\eqref{square} to compute
$\phi^2$ we learn
\begin{align*}
-\g{[u, \phi(v)], w} &+ \alpha(w)\g{[u, \phi(v)], R_\alpha} 
-\g{[\phi(w), u], v}\\ &+ \alpha(v)\g{[\phi(w), u], R_\alpha} + 
\g{[\phi(v), \phi(w)], \phi(u)}
= 0.
\end{align*}
This may be used to eliminate the first and third term in the preceding equation
which becomes
\begin{align}\label{eq:intermediate-phi}
2\g{(\nabla_u \phi) (v), w} &=
\g{[w, \phi(v)], u} + \g{[\phi(w), v], u} - \g{[v, w], \phi(u)} 
+ \alpha(w)\g{[u, \phi(v)], R_\alpha} \\ 
\notag &\qquad + \alpha(v)\g{[\phi(w), u], R_\alpha} + \g{[\phi(v), \phi(w)], \phi(u)}.
\end{align}
Note that 
\begin{equation}\label{eq:R-dalpha-theta'}
\begin{split}
\g{[u, \phi(v)], R_\alpha} & = \alpha([u, \phi(v)]) = -d\alpha(u, \phi(v))\\
& = -d\alpha(u^\xi, J v^\xi) = -\theta'\g{u^\xi, v^\xi}.
\end{split}
\end{equation}
Using this and the analogous formula with $w$ instead of $v$ allows to rewrite
Equation~\eqref{eq:intermediate-phi} as:
\begin{equation}\label{eq:next-intermediate-phi}
\begin{split}
2\g{(\nabla_u \phi) (v), w} &=
\g{[w, \phi(v)], u} + \g{[\phi(w), v], u} - \g{[v, w], \phi(u)} +
\g{[\phi(v), \phi(w)], \phi(u)} \\
&\quad - \theta'\g{u^\xi, \left(Id \wedge \alpha\right) (v, w)} 
\end{split}
\end{equation}
We want to rewrite the first line of this equation in terms on $[J, J]$.
In the following computation we use the definition of $[J, J]$, the fact
that, for any $u$,  $u^\xi = u - \alpha(u)R_\alpha$, $\phi(u) = Ju^\xi$ 
and the fact that $[Jv^\xi, Jw^\xi] - [v^\xi, w^\xi]$ is tangent to $\xi$ for any
vector fields $v$ and $w$.
\begin{align*}
[J, J](v^\xi, Jw^\xi) &= -[v^\xi, Jw^\xi] - [Jv^\xi, w^\xi] - J([Jv^\xi, Jw^\xi] - [v^\xi, w^\xi]) \\
&= -[v - \alpha(v)R_\alpha, \phi(w)] - [\phi(v), w - \alpha(w)R_\alpha] \\
&\quad -\phi\big([\phi(v), \phi(w)] - [v - \alpha(v)R_\alpha, w -
\alpha(w)R_\alpha]\big) 
\end{align*}
Next we use that $\alpha(v)$ and $\alpha(w)$ are constant to get them
out of Lie brackets and we use Equation~\eqref{eq:formulaforh} to get
\[
[J, J](v^\xi, Jw^\xi) =
-[v, \phi(w)] - [\phi(v), w] -\phi([\phi(v), \phi(w)]) + \phi([v, w]) \\
 -2 \left(h \wedge \alpha\right) (v, w).
\]
Comparing the above with Equation~\eqref{eq:next-intermediate-phi} and
noticing that both $[J, J]$ and $h \wedge \alpha$ take values in $\xi$
gives the announced Equation~\ref{eqn:nabla-phi}.
\end{proof}

\begin{proof}[Proof of Proposition~\ref{3Dcurvature}]
Because $h$ restricted to $\xi$ is symmetric with respect to $g$, it is
diagonalizable in a orthonormal basis. Assume that we have a local vector field 
$v$ with $h(v)=\lambda v,$ for some function $\lambda.$ Since
$h$ anti-commutes with $J$ we see $h(Jv)=-\lambda Jv.$ So locally we can get an
orthonormal frame $\{v_1, Jv_1,\ldots, v_n, Jv_n\}$  for $\xi$ such that
$hv_i = \lambda_i v_i$ and $hJv_i = -\lambda_i Jv_i$ for some non-negative
functions $\lambda_i$.

Let $v=v_i$, and $\lambda=\lambda_i$ for some $1\leq i\leq n$. From
Equation~\eqref{eqn:nabla_vN} we have
\begin{gather*}
\nabla_v R_\alpha =(\theta'/2-\lambda)Jv, \\
\nabla_{Jv} R_\alpha = -(\theta'/2 + \lambda) v.
\end{gather*}
So, for any vector $u$ in $\xi$ we can write 
$u = \sum (a_i v_i + b_i Jv_i)$ and the above equation gives
\[
\|\nabla_u R_\alpha\|^2 = 
\sum\left(\left(\frac{\theta'}2 + \lambda_i\right)^2b_i^2 +
\left(\frac{\theta'}2 - \lambda_i\right)^2a_i^2 \right)
\]
so, using also Equation~\eqref{eqn:nabla_NN} claiming 
$\nabla_{R_\alpha} R_\alpha = 0$ and the fact that each $\lambda_i$ is
non-negative and $\theta'$ is positive, we get
$\|\nabla R_\alpha\| = \theta'/2 + \max_i\lambda_i$.

Let $Q$ be the quadratic form on $\xi$ defined by
\[
	Q(u) = \g{\Rm(u, R_\alpha)R_\alpha, u} + \g{\Rm(Ju, R_\alpha)R_\alpha, Ju}
\]
where $\Rm$ is the Riemann curvature tensor defined by
$\Rm(X, Y)Z = \nabla_X(\nabla_Y Z) - \nabla_Y(\nabla_X Z) - \nabla_{[X, Y]}Z$
so both terms are sectional curvatures if $\|u\| = 1$. Let $u$ be a vector
field tangent to $\xi$. Because of Proposition~\ref{Rderivative} and the
symmetry of the Levi-Civita connexion, $[R_\alpha, u]$ is also tangent to $\xi$.
In the following computation we get rid of Lie brackets using that the
Levi-Civita connexion  is torsion free and transform all terms of type
$\nabla_v R_\alpha$ using Proposition~\ref{Rderivative} and
Equation~\eqref{eqn:nabla_vN}. We also use that $h$ and $J$ anticommute
(Equation~\eqref{eqn:JhhJ}) and Equation~\eqref{eqn:nabla_NJ}.
\begin{align*}
\g{\Rm(u, R_\alpha)R_\alpha, u} &= 
\g{\nabla_{[R_\alpha,u]} R_\alpha, u} - \g{\nabla_{R_\alpha}(\nabla_u R_\alpha), u}
+ \g{\nabla_u(\nabla_{R_\alpha} R_\alpha), u}\\
&= \g{J(\theta'/2-h)[R_\alpha, u], u}
- \g{\nabla_{R_\alpha}(J(\theta'/2-h)u), u} \\
&= \g{J(\theta'/2-h)\nabla_{R_\alpha} u + ((\theta'/2)^2-h^2)u, u}
- \g{\nabla_{R_\alpha}(J(\theta'/2-h)u), u} \\
&= \g{((\theta'/2)^2-h^2)u - (\nabla_{R_\alpha}h)Ju, u}.
\end{align*}
Adding the same formula applied to $Ju$ leads to
\[
	Q(u) = 2\g{((\theta'/2)^2-h^2)u, u}.
\]
So $Q$ is a quadratic form on $\xi$ having eigenvalues (with respect to the
inner product $g$) $\mu_i = 2((\theta'/2)^2-\lambda_i^2)$. The minimum
of $Q$ on the unit sphere is 
$\min(\mu_i) = 2(\theta'/2)^2 - 2\max(\lambda_i)^2 $.

So we proved
\[
	\max \lambda_i = \sqrt{\left(\frac{\theta'}2\right)^2 - \frac12  \min_{u \in \xi,
	\|u\| = 1} Q(u)}
\]
(in particular, the term under the square root is non-negative).
This gives the formula we wanted for $\|\nabla R_\alpha\|$.

Regarding Equation~\eqref{eq:nabla-phi-norm}, we note that
Equation~\eqref{eqn:nabla-phi} gives:
\[
|\g{(\nabla_u \phi)(v), w}| \leq 
\half\left(\|[J, J]\| + 2\|2h - \theta'Id\|\right)\,\|u\|\cdot\|v\|\cdot\|w\|.
\]
and clearly $\|2h - \theta'Id\| = 2\max \lambda_i + \theta'$ so we have
the announced bound.
\end{proof}

\section{A tightness radius estimate}
\label{S:background}
\label{sec:tauPSestimate}

This section is devoted to the proof of Theorem \ref{thm:main-higher}. 
It uses standard holomorphic curves arguments and a key comparison of
Riemannian and almost complex convexity in symplectizations of contact metric
manifolds. This comparison is explained in Subsection~\ref{ss:contactconvexity} after we
recall a few results about Riemannian convexity in Subsection~\ref{ss:riemannianconvexity}. We
then recall the definition of $PS$-overtwisted manifolds and their relevant
properties in Subsection~\ref{SS:PS} before proving the theorem in Subsection~\ref{SS:proof-tight-radius}.

\subsection{Convexity in Riemannian geometry}\label{ss:riemannianconvexity}

Let $S$ be a cooriented hypersurface in a Riemannian manifold $(M^n,g)$. 
We say that $S$ is \dfn{geodesically convex at $p$} if its second fundamental
form is positive definite at $p$. This has to do with convexity because it
forces geodesics which are tangent to $S$ at $p$ to stay (locally) on one side
of $S$. Let $f$ be a function defined on a small neighborhood $U$ of $p$ such
that $S \cap U$ is a regular level set of $f$ and $\nabla f$ defines the coorientation
of $S$. Then $S$ is geodesically convex at $p$ if and only if the Hessian of
$f$ at $p$ is positive definite.

The hypersurfaces we will consider are geodesic spheres so the relevant
functions are distances from points. To any real number $k$, one associates
the reference function

\begin{equation}
\label{eqn:cotsdef}
 \begin{split}
  \ct_k(r) & =\begin{cases}
		\sqrt{k}\cot(\sqrt{k} r)\, , & \qquad \text{if $k > 0$}\\
  \frac{1}{r}, & \qquad \text{if $k = 0$}\\
  \sqrt{-k}\coth(\sqrt{-k} r), &\qquad \text{if $k < 0$.}
  \end{cases}\\
 \end{split}
\end{equation} 
\begin{proposition}[see e.g. {\cite[Theorem 27 page 175]{Petersen}}]
\label{prop:riem_convexity}
Let $(M^n,g)$ be a Riemannian manifold with $\sec(g)\leq K$ for some real
number $K$. Let $r$ be any radius below the injectivity radius $\injg$ and $p$
any point in $M$.
If $K$ is non-positive, then the Hessian of the distance function 
\[
\mathsf{r}_p:M^n\to \R: q\mapsto d(p,q)
\]
is positive definite on the ball of radius $r$ about $p,$ $B_p(r)$.
If $K$ is positive then the same holds provided $r$ is less than 
$\frac{\pi}{2\sqrt K}$.

More generally, the Hessian of $\mathsf{r}$ satisfies 
\[
\nabla^2\mathsf{r}_p \geq \ct_K(r) g
\]
and $\ct_K(r)$ is positive whenever $K$ is non-positive or $K$ is positive and 
$r < \frac{\pi}{2\sqrt K}$.
\qed
\end{proposition}
 
\no The convexity radius $\conv(g)$ of Theorem \ref{thm:main-higher},
 refers to the maximum $r>0$ such that all geodesic balls of radius less than or equal to $r$ have geodesically convex boundary as hypersurfaces.
\subsection{Pseudo-convexity in symplectizations}
\label{ss:contactconvexity}

We consider the setup analogous to the one in \cite{EKM-sphere} but in dimension $\geq 3$.  
The symplectization of a contact manifold $(M, \xi)$ equipped with a
distinguished contact form $\alpha$ is 
the product $W := \R\times M$, equipped with the symplectic form
$\omega = d(e^t\alpha),$ where $t$ is the coordinate on $\R$.
Let $J$ be a complex structure on $\xi$ that is compatible with $(d\alpha)|_\xi$
and extended to $TW$ by setting $J\partial_t=R_\alpha.$

Let $U$ be a regular sublevel set of some function $f:M\to \R$ and
$S=\partial U.$ We can think of $f$ as a function on $W$ (by composing with the
projection $W\to M$) and thus we get the regular sublevel set $\Omega=\R\times
U$ with boundary $\Sigma=\R\times S.$
The complex tangencies to $\Sigma$,
\[
 \mathcal{C}_\Sigma=T\Sigma\cap J(T\Sigma),
\]
 can  be described as the kernel of the 1--form
$
df\circ J.
$
The form 
\[
L(u,v)=-d(df\circ J)(u,Jv)
\]
is called the \dfn{Levi form} of $\Sigma$.  Recall that $\Sigma$ is said to be
\dfn{pseudoconvex}, or \dfn{strictly pseudoconvex} if $L(v,v)\geq0,$
respectively $L(v,v)>0,$ for all non-zero $v\in \mathcal{C}_\Sigma$.

We also extend the
metric $g$ on $M$ to $W$ by $g+dt\otimes dt.$ 
The following result which relates Riemannian and symplectic convexity should be compared with the analogous result in
K\"ahler geometry \cite[Lemma page 646]{GreenWu} and in dimension 3, \cite{EKM-sphere}.

\begin{proposition}
\label{prop:levi-form}
Let $(\alpha, g, J)$ be a compatible metric structure on the
contact manifold $(M^{2n+1},\xi).$ Then, using the notation above, for any $v \in \mathcal{C}_\Sigma$ we have
\[
L(v,v) = \nabla^2 f(v,v)+\nabla^2 f(Jv,Jv).
\]
\end{proposition}

The rest of this subsection is devoted to the proof of the preceding proposition
and can be safely skipped on first reading.
For convenience of notation we denote $R_\alpha$ by $n$ in the computations.  We define 
two bundle maps $A:\xi\to TW$ and $B:\xi\to TW$ by
\[
A(v) =J[Jv, v] - \nabla_v v - \nabla_{Jv} Jv,
\]
\no and
\[
B(v)  =J[v, n] + \nabla_{Jv} n + \nabla_n Jv.
\]
One can easily check that the values $A$ and $B$ at a point only depend on the
vector at the point and not on the local extension to
local vector fields (more precisely $B$ is a tensor and $A$ is a
quadratic bundle map). We begin with the following observation (which originally
appeared for metric contact 3-manifolds in \cite{EKM-sphere}).

\begin{lemma}
	\label{lemma:common}
Under the hypotheses of Proposition~\ref{prop:levi-form}, we have
\[
-L(v,v) +\nabla^2 f(v,v)+\nabla^2 f(Jv,Jv) = 
df\left(A(v^\xi) + a B(Jv^\xi) - b B(v^\xi) - (a^2 + b^2) \nabla_n n \right),
\]
where the vector $v\in \mathcal{C}_\Sigma$ is written as $v=v^\xi+a\, n+
b\, \partial_t,$ with $v^\xi\in \xi$.
\end{lemma}

\begin{proof}
We extend $v$ to a vector field on $W$ as $v=v^\xi+a\, n + b\, \partial_t,$
with $v^\xi\in \xi$ invariant under translation in the $\R$ direction, and $a$
and $b$ are constants.
We first compute
\begin{equation*}
\begin{aligned}
\nabla^2 f (v,v)+\nabla^2 f (Jv,Jv)&= v\cdot (df(v))-(\nabla_v v)\cdot f
+ (Jv)\cdot (df(Jv)) - (\nabla_{Jv}Jv)\cdot f\\
&=v\cdot (df(v)) + (Jv)\cdot (df(Jv)) - df(\nabla_v v +\nabla_{Jv}Jv).
\end{aligned}
\end{equation*}
And, using the formula $d\alpha(u,w) = u\cdot\alpha(w) - w\cdot\alpha(u) -
\alpha([u, w])$ we have
\begin{equation*}
d (df\circ J) (v, Jv) = - v\cdot (df(v)) -(Jv)\cdot (df(Jv)) + df(J [Jv,v]).
\end{equation*}
Adding the two preceding equations, we obtain
\[
-L(v,v) +\nabla^2 f(v,v)+\nabla^2 f(Jv,Jv)= df(J [Jv,v]- \nabla_v v -\nabla_{Jv}Jv).
\]
\no Decomposing $v$ as $v^\xi+a\, n +b\partial_t$ as in the statement of the lemma
and using $\nabla_{\partial_t} v = 0$ we compute
\begin{align*}
J[Jv, v] & = J[Jv^\xi, v^\xi] + a J[Jv^\xi, n] + b J[n,v^\xi], \\
  \nabla_v v & = \nabla_{v^\xi} v^\xi + a(\nabla_{v^\xi} n + \nabla_n v^\xi) +
a^2 \nabla_n n, \\
  \nabla_{Jv} Jv & = \nabla_{Jv^\xi} Jv^\xi + b(\nabla_{Jv^\xi} n + \nabla_n Jv^\xi)  +
b^2 \nabla_n n.
\end{align*}
Thus we see that $-L(v,v) +\nabla^2 f(v,v)+\nabla^2 f(Jv,Jv)$ equals
\begin{equation*}
df\left(A(v^\xi) + a B(Jv^\xi) - b B(v^\xi) - (a^2 + b^2) \nabla_n n\right),
\end{equation*}
giving the announced formula.
\end{proof}

Now we establish Proposition~\ref{prop:levi-form}. We
begin by computing an expression for the operators $A$ and $B$ from
Lemma~\ref{lemma:common}.

\begin{lemma}
	\label{lemma:computeAB}
Under the hypotheses of Proposition~\ref{prop:levi-form}, for any vector $v$ in
$\xi$ we have
\[
A(v) = -\theta'\|v\|^2 \, \partial_t
\quad \text{ and } \quad
B(v) = -\theta'\,v.
\]
\end{lemma}
\begin{proof}
We can use
Equations~\eqref{eqn:nabla_vN}, \eqref{eqn:JhhJ}, and~\eqref{eqn:nabla_NJ} to
compute $B(v)$
\begin{align*}
B(v)&=J[v, n] + \nabla_{Jv} n + \nabla_n Jv = J\nabla_v n - J \nabla_n v  +
\nabla_{Jv} n + \nabla_n Jv \\
&=J\nabla_v n + \nabla_{Jv} n 
= (-\theta'/2\, v + h(v)) +J(\theta'/2\, Jv - h(Jv))=-\theta'\, v.
\end{align*}
We now compute $A(v)$ by projecting it to $n$, $\partial_t$ and $\xi$. Starting
with the projection to $n$ we use Equation~\eqref{eqn:trII} to conclude that
\begin{align*}
  \g{A(v), n} &= \g{J[Jv, v],n} - \g{\nabla_{v} v, n} -
  \g{\nabla_{Jv} Jv, n} \\
  &= \g{[Jv, v],\partial_t} - \II(v,v) - \II(Jv, Jv) = 0.
\end{align*}
Continuing with the projection to $\partial_t$ we have
\begin{align*}
\begin{split}
  \g{A(v), \partial_t} &= \g{J[Jv, v], \partial_t} =-\g{[Jv, v], n}\\
  & = -\g{\nabla_{Jv} v, n} + \g{\nabla_{v} Jv, n}  = \g{v, \nabla_{Jv} n +  J\nabla_{v} n}.
\end{split}
\end{align*}
From the last line of our computations of $B(v)$ above we conclude that $
\g{A(v), \partial_t} = -\theta'\|v\|^2.$ Finally, we can compute, for
any $w$ in $\xi$
\begin{align*}
\g{A(v), w} &= \g{J[Jv, v] -\nabla_{v} v - \nabla_{Jv} Jv, w}
=\g{\phi([\phi(v), v]) -\nabla_{v} v - \nabla_{\phi(v)} \phi(v), w}\\
&= \g{\phi(\nabla_{\phi(v)} v) - \phi(\nabla_{v} \phi(v))  
-\nabla_{v} v - \nabla_{\phi(v)} \phi(v), w} \\
&= -\g{(\nabla_{\phi(v)} \phi)(v), w} + \g{(\nabla_v \phi)(\phi(v)), w}\\
&= -\half\g{[J, J](v, Jw), Jv} + \half\g{[J, J](Jv, Jw), v}
\qquad\text{by Equation~\eqref{eqn:nabla-phi}}\\
&= \half\g{J[J, J](v, Jw) + [J, J](Jv, Jw),\; v}
\end{align*}
Using the definition of $[J, J]$ and the fact that 
$J^2 = -id$, one can check that, 
for any $v$ and $w$ in $\xi$,
$J[J, J](v, Jw) + [J, J](Jv, Jw) = 0$. Hence $A(v)$ has no component in $\xi$.
\end{proof}

\begin{proof}[Proof of Proposition~\ref{prop:levi-form}]
Combining Equation~\eqref{eqn:nabla_NN} and Lemmas~\ref{lemma:common}
and~\ref{lemma:computeAB} we see that
\begin{align*}
-L(v,v) +\nabla^2 f(v,v)+\nabla^2 f(Jv,Jv) &= 
df\Big(-\theta'(\|v^\xi\|^2 \partial_t + a Jv_\xi - b v_\xi)\Big) \\
&= -\theta'df(\|v\|^2 \partial_t + a Jv - b v) =0,
\end{align*}
where the last equality follows since $v\in \mathcal{C}_\Sigma.$
\end{proof}

\subsection{Bordered Legendrian open books}
\label{SS:PS}
Let $N$ be a compact manifold with nonempty boundary.  A \dfn{relative open
book} on $N$ is a pair~$(B, \theta)$ where
\begin{itemize}
\item 
the \dfn{binding} $B$ is a nonempty codimension~$2$ submanifold in the
interior of $N$ with trivial normal bundle, and
\item 
$\theta \colon N \setminus B \to S^1$ is a fibration whose fibers are
transverse to $\partial N$, and which coincides in a neighborhood $B \times D^2$ of 
$B = B \times \{0\}$ with the normal angular coordinate.
\end{itemize}

\begin{definition}[Massot, Niederkr\"uger and Wendl 2013, \cite{MNW}]
Let $(M,\xi)$ be a $(2n+1)$-dimensional contact manifold.  A compact 
$(n + 1)$-dimensional submanifold $N\hookrightarrow M$ with boundary is called
a \dfn{bordered Legendrian open book} (abbreviated \dfn{\BLOB{}}), if it has a
relative open book $(B, \theta)$ such that
\begin{itemize}
\item [(i)] all fibers of $\theta$ are Legendrian, and 
\item [(ii)] the boundary of $N$ is Legendrian.
\end{itemize}
If such a submanifold exists then $(M, \xi)$ is called \dfn{PS-overtwisted}.
\end{definition}
We notice that the notion of a {\em plastikstufe} defined in \cite{Niederkrueger06} 
is a special case of a \BLOB where the fibers of the \BLOB are of the form $B\times [0,1]$. 
The term PS-overtwisted originally referred to the existence of a plastikstufe in a contact
manifold, but it was generalized in \cite{MNW}.  
Although it is not certain if this definition is a definitive generalization of
overtwisted to higher dimensional manifolds since it may or may not be
equivalent to the definition in \cite{BormanEliashbergMurphy}, it does have
some of the properties of 3 dimensional overtwisted contact manifolds. In
particular, we have the following result. 

\begin{theorem}[Niederkr\"uger 2006, \cite{MNW, Niederkrueger06}]
If $(M,\xi)$ is a PS-overtwisted contact manifold then it cannot be symplectically
filled by a semi-positive symplectic manifold. 
If the dimension of $M$ is less than 7 then it cannot be filled by any
symplectic manifold. \qed
\end{theorem} 

A $2n$--dimensional symplectic manifold $(X,\omega)$ is called semi-positive if
every element $A\in \pi_2(X)$ with $\omega(A)>0$ and $c_1(A)\geq 3-n$ satisfies
$c_1(A)>0.$ Note that all Stein and exact symplectic manifolds are
semi-positive.

The presence of a \BLOB also has dynamical consequences and they will be crucial
in our proof of Theorem \ref{thm:main-higher}. Recall that the Weinstein
conjecture asserts that any Reeb vector field on a closed contact manifold has
some closed Reeb orbits. Contractible Reeb orbits do not always exist
but the considerations in \cite{MNW} allows to slightly generalize the main
theorem in \cite{AlbersHofer09} resulting in the following theorem.
\begin{theorem}[Albers and Hofer 2009, \cite{AlbersHofer09}]
\label{albershofer}
Let $(M,\xi)$ be a closed  PS-overtwisted contact manifold. Then every Reeb
vector field associated to $\xi$ has a contractible periodic orbit.\qed
\end{theorem}

\subsection{Proof of the tightness radius estimate}
\label{SS:proof-tight-radius}

We can now prove Theorem~\ref{thm:main-higher} which claims that a ball $B(x, r)$ whose
radius $r$ is below the convexity radius in a contact metric manifold cannot contain
a \BLOB.

\begin{proof}[Proof of Theorem~\ref{thm:main-higher}]
If we assume the existence of a \BLOB, then one can start a family of
holomorphic disks as in \cite{Niederkrueger06}. Because of the Levi form
computation of Proposition \ref{prop:levi-form}, the boundary of a convex ball
lifts to a pseudo-convex hypersurface in the symplectization. 
The weak maximum principle for elliptic operators then guarantees that
holomorphic curves cannot ``touch from the inside" this hypersurface. 
This allows to use the strategy of \cite{AlbersHofer09} without any modification
and prove the existence of a closed Reeb orbit $\gamma$ inside the ball $B$.

However, such an orbit would be a closed geodesic according to Proposition
\ref{Rderivative}. Those cannot exist inside $B$ because it would have to be
somewhere tangent to a sphere $S(x, r_0)$, for some $r_0$, with $\gamma$ lying
inside the ball $B(x, r_0).$  Of course $r_0$ is also below the convexity radius
so $\partial B(x, r_0)$ cannot be ``touched from the inside'' by a geodesic and
we get a contradiction.

The part of Theorem \ref{thm:main-higher} relating to curvature follows from the
above and the estimate on the Hessian of the radial function given in
Proposition~\ref{prop:riem_convexity}.
\end{proof}

\begin{proof}[Proof of Corollary~\ref{cor:ut-higher}]
Note that pull-backs of $\xi$ and the metric to any covering space are
compatible and the sectional curvature is non-positive. It is well known, by
Hadamard Theorem \cite[Theorem IV.1.3 page 192]{Chavel06}, that the universal
cover of a manifold with nonpositive curvature is $\R^{2\,n+1}$ and the space is
exhausted by geodesic balls.  Moreover, the convexity radius of the universal
cover is infinite. Thus Theorem~\ref{thm:main-higher} says that a ball of any
radius is \BLOB free. The result follows.
\end{proof}

\section{A quantitative Darboux theorem in any dimension}
\label{sec:darboux}

\no 
In this section, we establish an estimate on the Darboux radius of a contact
manifold with a compatible metric structure. We begin by introducing a
number of quantities used throughout this section that depend on the
dimension, the instantaneous rotation $\theta'$ of the contact structure and
bounds on curvature and injectivity radius. Unless said otherwise, we assume that the
sectional curvature of $g$ is between $\kappa$ and $K$:
\begin{equation}\label{eq:kappa-sec-K}
 \kappa\leq \text{\rm sec}(P)\leq K,
\end{equation}
for any 2-plane $P$ in $TM$. Further, we define
\begin{equation}\label{eq:r-interval}
	\rmax = 
	\begin{cases}
		\min\left(\injg,\frac{\pi}{2\sqrt{K}}\right) &\text{if $K > 0$}\\
		\injg &\text{if $K \leq 0$}.\\
	\end{cases}
\end{equation}

We also define the quantities
\begin{align}
	\label{eq:def-AB}
A &= \A \qquad \text{ and}\nonumber\\
B &= \B\nonumber
\end{align}
where $|\text{\rm sec}(g)|$ is the maximum in absolute value sectional curvature
over $(M,g)$, and $K(u, v)$ denotes the sectional curvature of the plane
spanned by $u$ and $v$. The square root appearing in $B$ is well defined thanks
to Proposition~\ref{3Dcurvature}.  (Note than $A$ and $B$ have nothing to do
with the operators appearing in the Levi form  computation of the preceding
section.)

In addition to the reference function $\ct_k$ defined by
Equation~\eqref{eqn:cotsdef}, we will need functions $\sn_k$, also indexed
by a real number $k$
\begin{equation}
 \begin{split}
  \sn_k(r) & =\begin{cases}
		\frac{1}{\sqrt{k}}\sin(\sqrt{k} r)\, , & \qquad \text{if $k > 0$}\\
   r, & \qquad \text{if $k = 0$},\\
	 \frac{1}{\sqrt{-k}}\sinh(\sqrt{-k} r), &\qquad \text{if $k < 0$.}
  \end{cases}\\
 \end{split}
\end{equation} 

These functions combine with $A$ and $B$ and an upper bound $K$ on the
sectionnal curvature to define
\[
Q(r) := \sn_K^{-1}\left((1 - B r-\frac{1}{2} A r^2)\sn_K(r)\right).
\]
and the constants 
\begin{align}
 \overline{H}_1 &=\begin{cases}
	 \sqrt{2},&  \text{if}\ \kappa\geq 0,\\
   \sqrt{1 + \left(\frac{\sn_\kappa(\rmax)}{\rmax}\right)^2} & \text{if}\  \kappa<0,
\end{cases},\nonumber\\ 
\overline{H}_2&=\frac{8}{3}|\sec(g)|\overline{H}_1 \rmax,\,\text{ and }\label{eq:H1-H2-bar-consts}\\ 
\overline{H}&= 2\overline{H}_1\overline{H}_2+(\|[J, J]\|/2 +2B)\overline{H}_1^2.\nonumber
\end{align}

Combining all those numbers, and using the fact that $B$ is positive, we define
\begin{equation}\label{eq:rtamed}
	\rtamed = \min\left(\frac{\injg}{2}, \frac{\pi}{2\sqrt{K}}, \frac{2}{\sqrt{2\,A+B^2}+B}, 
	\frac 1{(1+2n(n-1))\overline{H}}\right) 
\end{equation}
We are now ready to state the refined version of
Theorem~\ref{thm:GeometricDarbouxHigherRough}, with a refined estimate, which
will be proven in this section. 
\begin{theorem}
	\label{thm:GeometricDarbouxHigher}
 Given a compatible metric structure $(\alpha, g, J)$  on  $(M,\xi)$, with sectional curvature bounded as in \eqref{eq:kappa-sec-K}, the Darboux radius admits the following  bound
\begin{equation}\label{eq:delta}
 \delta(M,\xi)  \ge Q(\rtamed).
\end{equation}
\end{theorem}

First we explain how the coarser version announced in the introduction follows
from the above result.
\begin{proof}[Proof of Theorem~\ref{thm:GeometricDarbouxHigherRough}]
We now discuss how this bound simplifies if we are willing to assume that the
sectional curvature of $g$ is in $[-K,K]$ for some positive constant $K$
and we roughly estimate the complicated functions appearing above.
We set 
\[
\bound = \max(\theta', \sqrt K, \|[J, J]\|).
\]
In particular, 	$A \leq 4\sqrt{2n - 1}\bound^2/3 \leq 2\sqrt{n}\bound^2$. 
Using that $\sqrt{(1 + x^2)} \leq 1 + x$ for
non-negative $x$ we estimate 
\begin{align*}
	B &\leq \frac{\theta'}2\left(1 + \sqrt{ 1 + \frac{4K}{\theta'^2}}\right) 
	\leq \frac{\theta'}2\left(1 + 1 + \frac{2\sqrt K}{\theta'}\right)\\
	&\leq 2\bound.
\end{align*}
And we compute
\begin{align*}
\overline{H}_1 &= 
\sqrt{1 + \left(\frac{\sn_{-K}(\rmax)}{\rmax}\right)^2} 
= \sqrt{1 + \left(\frac{\sinh(\sqrt{K}\rmax)}{\sqrt{K}\rmax}\right)^2} \\
&\leq 
\sqrt{1 + \left(\frac{2}{\pi}\sinh\left(\frac{\pi}{2}\right)\right)^2} < 2
\end{align*}
because $\sqrt{\M}\rmax$ is always less than $\pi/2$ (there are two cases to
check depending on what term attains the min in the definition of $\rmax$).
So
$\overline{H}_1 \leq 2$ and $\overline{H}_2 \leq \frac{8\pi\bound}{3}$ and
$\overline{H} \leq 52\bound$.

On the other hand, our estimates on $A$ and $B$ give:
\[
	\sqrt{2A + B^2} + B \leq 6n^{1/4}\bound.
\]

So, the terms appearing in the definition of $\rtamed$ are estimated as follows
\[
\frac{2}{\sqrt{2A + B^2} + B} \geq \frac1{3n^{1/4}\bound} \qquad
\text{and} \qquad
\frac 1{(1+2n(n-1))\overline{H}} \geq 
\frac{1}{104n^2\bound}.
\]
Note that the second bound above is always less than the first one so,
setting 
$d_n := 1/(104n^2)$, we see
\[
\rtamed \geq \min\left(\injg, 
\frac{d_n}{\bound}\right).
\]
It remains to estimate $Q(r_0)$ for $r_0 \leq d_n/\bound$. 
Our estimates on $A$ and $B$ give
\[
1 - Br_0 - \frac 12 A r_0^2 \geq 
1 - 2d_n -
d_n^2\sqrt{n} \geq \frac{97}{100}.
\]
Therefore
\[
Q(r_0) \geq \frac1{\sqrt K} \arcsin\left(\frac{97}{100}\sin(\sqrt K r_0)\right)
\geq \frac1{\sqrt K} \frac 12\sqrt K r_0 = \frac{r_0}{2},
\]
which yields the promised Darboux radius estimate.
\end{proof}

We will now prove Theorem~\ref{thm:GeometricDarbouxHigher} modulo a number of
propositions which will be proved in subsequent subsections.  The goal is to
embed a large geodesic ball in our contact manifold into the standard contact
$\R^{2n + 1}$. The later is the contactization of the standard Liouville
structure on $\R^{2n}$ and we will compare it to some contactization of a natural
exact symplectic manifold inside our given contact metric manifold $M$. Recall a
\dfn{Liouville manifold} is a pair $(W,\lambda)$ where $d\lambda$ is a
symplectic form on $W$ and $\lambda$ restricted to the boundary of $W$ is a
contact form for a positive contact structure. Also recall that the
\dfn{contactization} of an exact symplectic manifold $(W, \beta)$, and in
particular a Liouville manifold, is $\R \times W$ equipped with the contact
structure $\ker(dt + \beta)$.

Given any point $p$ in $M$ and the contact hyperplane $\xi_p$ at $p$, the
geodesic disk $\D(r)$ centered at $p$ of radius $r$ and tangent to $\xi_p$ is
given as the image of the restriction of the exponential map to the disk of radius $r$ in $\xi_p$, that is
\[
\D(r) = \exp_p(D_\xi(r)).
\]
where $D_\xi(r)=\Bigl(\{ v \in \xi_p;\; |v| < r\}\Bigr).$
Denoting the Reeb flow by $\Phi(t,\mathbf{x}):\R\times M\longrightarrow M$ we define the map
\[
E:\R\times D_\xi(r)\to M: (t,v)\mapsto \Phi(t, \exp_p(v)),
\]
and the $R_\alpha$--invariant ``cylindrical'' neighborhood $\C(r)$ of $\D(r)$ 
to be the image of $E$. Of course $\C(r)$ is not, in general, an embedded submanifold 
of $M$, but for $r$ small enough $\D(r)$ will be an embedded disk and $R_\alpha$
will be transverse to $\D(r)$. For such an $r$, $\C(r)$ will  then contain
embedded neighborhoods of $\D(r)$, for example $E((-\epsilon, \epsilon)\times
D_\xi(r))$, for sufficiently small $\epsilon$. To prove
Theorem~\ref{thm:GeometricDarbouxHigher} we will proceed in the following steps.

\no {\bf Step I.} Find an estimate on the radius $r$ so that $R_\alpha$ is transverse to $\D(r)$. \\
{\bf Step II.} Find an estimate on the radius $r$ so that the pull back of the contact structure $\xi$ to $\R\times D_\xi(r)$ via $E$ embeds into the standard contact $\R^{2n+1}$. \\
{\bf Step III.} Find an estimate on the size of a geodesic ball about $p$ that embeds in $M$ and is contained in $\C(r)$.  

We will first list several propositions, that will be proven in the following subsections, that give the estimates indicated in the outline above and then assemble them into a proof of Theorem~\ref{thm:GeometricDarbouxHigher}. The estimate in Step I is given in the following proposition.

\begin{proposition}[proved in Section~\ref{S:t_p-est}]
	\label{prop:t-estimate}
	 Given a compatible metric structure $(\alpha, g, J)$  on the contact
manifold $(M,\xi)$, the disk $\D(r_0)$ is embedded and 
the Reeb vector field $R_\alpha$ is positively transverse to it if
\begin{equation}\label{eqn:main-bound-t}
	r_0 < \rtransverse := \min\left\{\injg, \frac{\pi}{2\sqrt{K}},\frac{2}{\sqrt{2\,A+B^2}+B}\right\},
\end{equation}
where the constants are defined at the beginning of this section. 
Moreover, if $\nD$ is a unit normal vector to $\D(r_0)$, then along any radial geodesic $\gamma=\gamma(r)$
\begin{equation}
\label{eq:gRnD}
\g{R_\alpha(r), \nD(r)} \geq 1 - Br -\frac 12 Ar^2,
\end{equation}
\end{proposition}

To carry out Step II we first make an observation about contactizations of
Liouville domains and exact symplectic manifolds.  For the remainder of this
section  $(W, \beta_0)$ will be a Liouville domain.  Let $\mu$ denote the
restriction of $\beta_0$ to $\partial W$. By definition $\mu$ is a contact form.
The completion of $W$ is obtained as usual by adding the cylindrical end $[1,
\infty) \times \partial W$ equipped with the Liouville form $t\mu$, where $t$ is
the ``radial'' coordinate on $[1,\infty)$. The resulting manifold will be
denoted $W_\infty$ and we will also denote this extended 1-form by $\beta_0$.
For any constant $a>1$ we set $W_a = W \cup \big([1, a) \times \partial W\big)$.
We say an almost complex structure is \dfn{adapted} to $\beta_0$ if
\begin{itemize}
\item[$(a)$] 
it is tamed by $d\beta_0$,
\item[$(b)$]
it preserves the contact structure $\ker \alpha$ on each 
$\{t\} \times \partial W$, and
\item[$(c)$]
it sends $\partial_t$, point-wise, to some positive multiple of the Reeb field $R_\mu$.
\end{itemize}

Recall that a 2--form $\omega$ tames an almost complex structure $J$ if
$\omega(u, Ju) > 0$ for any non-zero vector $u$. Note that $\omega$ is
then automatically non-degenerate since any $u$ in the kernel of $\omega$ 
would violate the taming condition. 

\begin{proposition}[proved in Section~\ref{sec:contactizations}]
	\label{prop:embeddings}
Suppose $\beta_1$ is a 1--form on $W_T$ (for some $T>1$) such that
$d\beta_1$ is a symplectic form on $W_T$ and 
there is an almost complex structure which is adapted to $\beta_0$
and tamed by $d\beta_1$. Then, for any $T_0 \in [1, T)$,
the contactization of $(\overline{W}_{T_0}, \beta_1)$ embeds in the
contactization of $(W_\infty, \beta_0)$.
\end{proposition}

In our situation, we want to apply the above proposition to the 
complex structure on $\D(r)$ obtained by pushing forward, via $\exp_p$, some complex
structure on $\xi_p$ tamed by $d\alpha_p$.

\begin{proposition}[proved in Section~\ref{sec:gettingJ}]
\label{prop:Jtamed}
 Given a compatible metric structure $(\alpha, g, J)$ on the contact
manifold $(M,\xi)$, with sectional curvature bounded as in \eqref{eq:kappa-sec-K}, 
the  complex structure $(\exp^\xi_p)_* J_p$ is tamed by the restriction of
$d\alpha$ to $\D(r)$ whenever 
\[
r<\min\left(\rtransverse,\frac 1{(1+2n(n-1))\overline{H}}\right),
\]
where the constants are defined at the beginning of this section. 
\end{proposition}

The previous two propositions will guaranty that the pull back of the contact
structure on $\C(r)$ via $E$ will be standard, that is embed in the standard
contact structure on $\R^{2n+1}$, thus completing Step II. So we are left to
complete Step III by estimating the size of a geodesic ball that can be embedded
in such a cylinder. We can make such an estimate in a more general context that
does not involve anything from the special geometry of compatible metrics except
that the Reeb field is geodesic.

\begin{proposition}[proved in Section~\ref{sec:cyl2ball}]
\label{prop:cyl2balls}
Let $(M, g)$ be a complete Riemannian manifold whose sectional curvature is
bounded above by $K$. Let $X$ be a unit norm geodesic vector field on $M$ and
$p$ a point in $M$. Consider the disk
\[
\D(r_0) := \{\exp_p(v) :\; v \in X_p^\perp, \|v\| < r_0 \}%
\quad \text{with}\quad%
r_0 < \min\left(\frac{\injg}2, \frac{\pi}{2\sqrt{K}}\right).
\]
We denote by $n$ a unit vector field positively transverse to $\D(r_0)$
and assume we have the following estimate along a radial geodesic
$\gamma$ 
\begin{equation}\label{eq:angle-est}
\g{X(\gamma(r)), n(\gamma(r))} \geq 1 - P(r),
\end{equation}
where $P=P(r)\geq 0$ depends only on the distance $r$ to $p$ and $P(r)\leq 1$ on $[0,r_0]$. 
Then the cylinder $\C(r_0)=\Phi\bigl( (-\infty, \infty) \times \D(r_0)\bigr)$ given by the 
flow $\Phi$ of $X$ contains a geodesic ball of radius 
\[
\sn_K^{-1}\big((1 - P(r_0)) \sn_K(r_0)\big)
\]
about $p$.
\end{proposition}

We can now prove Theorem~\ref{thm:GeometricDarbouxHigher} estimating the size of
a Darboux ball.

\begin{proof}[Proof of Theorem~\ref{thm:GeometricDarbouxHigher}]
By Proposition~\ref{prop:t-estimate}, if $r < \rtransverse$ then  
then $\D(r)$ is embedded in $M$ and the Reeb vector field $R_\alpha$ is
transverse to $\D(r)$. Since $R_\alpha$ is transverse to $\D(r)$ the restriction
of the contact form $\alpha$ to $\D(r)$ is a primitive for an exact symplectic
form $d\alpha$ on $\D(r)$. 

Let $\beta$ denote the pull back of $\alpha|_{\D(r)}$ to $D_\xi(r)$ by the
exponential map. The contactization of $(D_\xi(r),\beta)$ is the contact
structure on $\R\times D_\xi(r)$ coming from the  contact form $dt+ \beta$ and
$E:\R\times D_\xi(r)\to M$ is a contact immersion.  In fact $E^\ast\alpha=dt+\beta$ since 
$dE(\partial_t)=R_\alpha$, $\mathcal{L}_{R_\alpha}\alpha=0$, and $\alpha(R_\alpha)=1$, thus 
$\xi=\ker\alpha$ lifts to $\ker(dt+\beta)$.

The bound of Proposition \ref{prop:Jtamed} implies that if 
\begin{equation}\label{eq:db-est2}
r < \rtamed := \min\left(\rtransverse,\frac 1{(1+2n(n-1))\overline{H}}\right),
\end{equation}
then $d\alpha$ on $\D(r)$ is tamed by $J=(\exp^\xi_p)_\ast J_p$. Thus $\beta=(\exp^\xi_p)^*\alpha|_{\D(r)}$ on $D_\xi(r)\subset \xi_p$ is tamed by $J_p$. Of course the standard Liouville form $\lambda=\sum_{i=1}^n y_i\, dx_i$ is adapted to $J_p$. Thus by Proposition~\ref{prop:embeddings} the contactization of any open subdomain of $D_\xi(r)$ will contact embed in the contactization of $(\R^{2n},\lambda)$, that is in the standard contact structure on $\R^{2n+1}$. 

We are left to estimate the maximal size of a geodesic ball $B_p(r)$ about $p$
that can be embedded inside the cylinder $\C(r)$ (which of course can then be
lifted via $E$ to the contactization of $(D_\xi(r), \beta)$, as above.)
Setting $P(r)=Br+\frac 12 Ar^2$ we see from Equation~\eqref{eq:gRnD} that 
\[
\langle R_\alpha(r),n_{\mathsf{D}}(r)\rangle \geq 1- P(r)
\]
and Proposition~\ref{prop:cyl2balls} allows to embed a geodesic ball of
radius $Q(\rtamed)$ inside the cylinder $\C(\rtamed)$ hence this ball is
standard. 
\end{proof}

\subsection{Twisting estimates.}\label{S:t_p-est}
Throughout this subsection we will assume that
$r \in [0,\rmax)$ and use the notation established at the beginning of Section \ref{sec:darboux}. 
Our goal is to prove Proposition~\ref{prop:t-estimate} which estimates, along a radial geodesic, 
the angle $\g{R_\alpha, \nD}$
between the Reeb vector field and the disks $\D(r)$ as
\[
\g{R_\alpha, \nD} \geq 1 - Br -\frac 12 Ar^2,
\]
where $\nD$ is the unit normal vector to $\D(r)$ which coincides with $R_\alpha$
at $p$.

First notice that, since $\rmax$ is less than the injectivity radius of
$g$, the disk $\D(r)$ is embedded for any $r \leq \rmax$.
Next we show that the estimate above implies the transversality result in the proposition.
The Reeb field $R_\alpha$ is transverse to $\D(r)$ as long as
$\g{R_\alpha, \nD}$ is positive.
Because the roots of $A\,t^2+2 B\,t-2$ are $(-B\pm\sqrt{2 A+B^2})/A$, this is
guaranteed whenever $r$ is less than $\frac{-B+\sqrt{2A+B^2}}{A}=\frac{2}{\sqrt{2\,A+B^2}+B}>0$. Then, we obtain
\[
	\min\left(\rmax, \frac{2}{\sqrt{2\,A+B^2}+B}\right) = \rtransverse.
\]

We now establish the estimate on $\g{R_\alpha, \nD}$. Let $q$ be any point of
$\D(r)$ at some distance $r_q$ from $p$. We denote by $\gamma$ the radial
geodesic between $p$ and $q$. Estimating along $\gamma$ we have
\begin{equation}\label{eqn:est1}
\begin{split}
|\langle R_\alpha, \nD\rangle_{q}-1| &= 
 |\langle R_\alpha, \nD\rangle_{q}-\langle R_\alpha,
 \nD\rangle_{p}| \\
  & = \left| \int_0^{r_q} \frac{d}{ds} \langle R_\alpha, \nD\rangle_{\gamma(s)}\, ds\right|
   = \left|\int^{r_q}_0 \nabla_{\dot{\gamma}} \langle  R_\alpha, \nD\rangle\, 
 ds\right|\\
  & \leq \int^{r_q}_0 |\langle \nabla_{\dot{\gamma}}R_\alpha, \nD\rangle|+
  |\langle R_\alpha,\nabla_{\dot{\gamma}} \nD\rangle | \, ds.
\end{split}
\end{equation}
In the first term, $|\g{\nabla_{\dot{\gamma}}R_\alpha, \nD}|$ is less than or equal to  
$\|\nabla_{\dot{\gamma}}R_\alpha\|$ and thus Equation~\eqref{eq:nabla-reeb} from 
Proposition~\ref{3Dcurvature} gives
\begin{equation}\label{eqn:est3}
 |\langle \nabla_{\dot{\gamma}}R_\alpha, \nD\rangle| \leq  B.
\end{equation}
Hence we are left to estimate $\g{R_\alpha(s), \nabla_{\dot\gamma} \nD(s)}$ for
all $s\in[0,r_q]$. We now fix such an $s$.
Because $s\leq r_q$ is less than the injectivity radius the Gauss lemma guarantees the
differential of the exponential map at $p$ gives us an isomorphism between the
orthogonal complement of $\dot\gamma(0)$ in $T_pM$ and the orthogonal
complement of $\dot\gamma(s)$ in $T_{\gamma(s)}M$. Using this and the fact that 
$\D(r)$ is the image of $D_\xi(r)$ under the exponential map one can
construct Jacobi fields $J_1, \dots, J_{2n - 1}$ along $\gamma$ such that
\begin{itemize}
	\item[$(i)$] $J_i(0) = 0$ and $J_i'(0)$ is in $\xi_p$,
	\item[$(ii)$] all $J_i$ are tangent to $D$, and
	\item[$(iii)$]  $\{\nD, \dot\gamma, J_1, \dots, J_{2n - 1}\}$ is an orthonormal basis
		at $\gamma(s)$ (but not, {\em a priori}, at any other $\gamma(r)$).
\end{itemize}
We now derive an estimate on the derivative of the Jacobi fields that will be needed below. 
\begin{lemma}
\label{lemma:nDJprime}
With the notation as above, 
let $J$ be any of the $J_i$ Jacobi fields. Then 
\begin{equation}
\label{eq:nDJprimeFinal}
\left|\g{\nD,J'(s)}\right| \leq \frac{4}{3}|\sec(g)| \,s.
\end{equation}
\end{lemma}

\begin{proof}
Integrating the Jacobi equations $J'' + R(J,\gamma')\gamma' = 0$ component-wise in a parallel
moving frame along $\gamma$, one obtains  
$J'(s) = J'(0) - \int_0^{s} R(J(t), \gamma'(t))\gamma'(t)\, dt$
where the integral is again component-wise integration in a parallel frame.
Hence, using $\g{\nD, J'(0)} = 0$ and the bound $|R| \leq \frac 43 |\sec(g)|$
(see, for example, \cite[page 95]{Chavel06}) one obtains
\begin{equation}
\label{eq:nDJprimeTemporary}
\left|\g{\nD,J'(s)}\right| \leq  \frac{4}{3}|\sec(g)| \max_t \|J(t)\| \,s.
\end{equation}

We claim that for $s\leq\min\bigl(\injg,\frac{\pi}{2\sqrt{K}}\bigr)$
the function 
$\|J(t)\|$  is increasing on the interval $[0,s]$. Indeed, if $\mathsf{r}$ is the
distance function from $p$ and $\nabla^2\mathsf{r}$ denotes its Hessian, then a
simple computation (or see \cite[Lemma III.4.10,
p. 109]{Sakai96}) coupled with 
Proposition \ref{prop:riem_convexity} yields
\begin{align*}
\frac{\partial}{\partial t} \|J(t)\|^2 & = 2g(\nabla_{\dot{\gamma}(t)}J(t),J(t))
=  2 \nabla^2 \mathsf{r} (J(t),J(t))\\
& \geq (\ct_K\circ\mathsf{r}(\gamma(t))\, \|J(t)\|^2
=\text{ct}_K(t)\, \|J(t)\|^2,
\end{align*}
and the length of $J(t)$ must increase until the first zero of $\ct_K(t)$,
which occurs at $t=\frac{\pi}{2\sqrt{K}}$ for $K>0$ and does not exist
otherwise.
Because $\|J(t)\|^2$ is increasing $\|J(t)\|\leq \|J(s)\|=1$ and Inequality~\eqref{eq:nDJprimeTemporary} simplifies to the promised Estimate~\eqref{eq:nDJprimeFinal}.
\end{proof}

At the point $\gamma(s)$, we can decompose the Reeb field as
$R_\alpha = R_\gamma \dot\gamma + R_n \nD + \sum R_i J_i$. 
The first two terms do not contribute to the scalar product with 
$\nabla_{\dot\gamma} \nD$ since $\nD$ is normal to $\gamma$ and $\gamma$ is a
geodesic. 
We can now estimate
\begin{align*}
|\g{R_\alpha, \nabla_{\dot\gamma} \nD}(\gamma(s))| &\leq  
\sum_{i = 1}^{2n - 1} |R_i \g{J_i(s), \nabla_{\dot\gamma} \nD}| 
= \sum_{i = 1}^{2n - 1} |R_i \g{\nabla_{\dot\gamma} J_i(s), \nD}|,
\end{align*}
where the last equality follows because $\g{J_i(t), \nD} = 0$ for all $t$.

Notice that $R_\gamma^2 + R_n ^2 + \sum R_i^2 = \|R_\alpha\|^2 = 1$ so 
$\sum R_i^2 \leq 1$.
Equation~\eqref{eq:nDJprimeFinal} from Lemma~\ref{lemma:nDJprime}
and the Cauchy-Schwarz inequality give:
\[
|\g{R_\alpha, \nabla_{\dot\gamma} \nD}(q)| \leq \sqrt{2n - 1}\,\frac{4}{3}|\sec(g)| s.
\]

Using this, Equation~\eqref{eqn:est1}, and  Equation~\eqref{eqn:est3} we see that
\begin{equation*}
\begin{split}
|\langle R_\alpha, \nD\rangle_{q}-1|
  & \leq \int^{r_q}_0 |\langle \nabla_{\dot{\gamma}}R_\alpha, \nD\rangle|+
  |\langle R_\alpha,\nabla_{\dot{\gamma}} \nD\rangle | \, ds\\
	&\leq Br_q+ \left( \sqrt{2n-1}\frac43 |\sec(g)| \right)\frac{r_q^2}2\\
  &= Br_q + \frac 12 Ar_q^2,
\end{split}
\end{equation*}
from which Equation~\eqref{eq:gRnD} easily follows.  \xqed

\subsection{Embedding contactizations} 
\label{sec:contactizations}
This subsection contains a proof of Proposition~\ref{prop:embeddings}.
We begin with a simple lemma about embedding contactizations. Through out this subsection we will be using notation established at the beginning of the section. 

\begin{lemma}[Interpolation lemma]
Let $(W,\beta_0)$ be a Liouville domain and $(W_\infty, \beta_0)$ its completion. 
Suppose $\beta_1$ is a 1--form on $W_T$ (for some $T>1$) such that
$d\beta_1$ is a symplectic form on $W_T$ and 
there is an almost complex structure which is adapted to $\beta_0$
and tamed by $d\beta_1$. 
Then for any $T_0\in [1,T)$ there is a positive constant $\lambda$ and a Liouville form $\widehat \beta$ on
$W_\infty$ such that 
\begin{itemize}
	\item[$(i)$] $\widehat{\beta} = \lambda\beta_0$ outside $W_T$,
	\item[$(ii)$] $\widehat{\beta} = \beta_1$ on $W_{T_0}$, and
	\item[$(iii)$] $d\widehat\beta$ is tamed by $J$.
\end{itemize}
\end{lemma}

\begin{proof}
We set $\widehat \beta = \rho \beta_1 + (1 - \rho)\lambda\beta_0$ where 
$\rho$ is a function with support in $W_T$, equals one on a neighborhood of 
$W_{T_0}$ and which, inside $[1, \infty) \times \partial W$, depends only on $t$
and is non-increasing.

The first two properties of $\widehat \beta$ are obvious from its definition, so we are left to show that
$\lambda$ and $\rho$ can be chosen so that the third property holds. Notice that for this we can
restrict our attention to 
$[1, \infty) \times \partial W$ where we have
\[
d\widehat\beta = (-\rho')\lambda t dt \wedge \mu + \rho' dt \wedge \beta_1
+ \rho d\beta_1 + (1 - \rho) \lambda d\beta_0.
\]
(Recall that $\mu$ is the contact form induced on $\partial W$ by
$\beta_0$.)
By the hypothesis $-\rho'$ is non-negative.
Moreover $dt \wedge \mu(u , Ju)$ is non-negative because $J$ is adapted 
to $\beta_0$.
So we are left to prove that, using carefully chosen $\rho$ and $\lambda$, 
for any non zero $u$ the function
\[
G(u) := \big(\rho' dt \wedge \beta_1 + 
\rho d\beta_1 + (1 - \rho) \lambda d\beta_0\big)(u , Ju)
\]
is positive.

Because $d\beta_0$ and $d\beta_1$ tame $J$, there are positive constants
$C_0$, $C_1$ and $C$ such that
\[
\frac 1{\|u\|^2}G(u)  \geq - |\rho'| C + \rho C_1 + \lambda (1 - \rho) C_0.
\]
We cut the interval $[T_0, T)$ in two halves. On the first interval, $\rho$ will be
almost 1 and $\rho'$ will be almost 0 and on the second interval $\rho$ will decrease to zero.
More precisely, we choose $\rho$ such that
\[
\rho(t) \geq 1 - \varepsilon \;\text{ and }\; 
	| \rho'(t)| \leq \frac{2\varepsilon}{(T - T_0 )/2},\qquad
	\text{ for } t \in [T_0, (T_0 + T)/2], 
\]
and
\[
	\rho(t) \leq 1 - \varepsilon \;\text{ and }\; 
	| \rho'(t)| \leq \frac{2(1 - \varepsilon)}{(T - T_0 )/2}, \qquad
	\text{ for } t \in [(T_0 + T)/2, T].
\]
On the first subinterval, one has
\[
- C |\rho'| + \rho C_1 \geq  
- C \frac{2\varepsilon}{(T - T_0 )/2} + (1 - \varepsilon)C_1
\]
which is positive if $\varepsilon$ is sufficiently small. On the second interval
\[
- C |\rho'| + \lambda(1 - \rho) C_0 \geq  
- C \frac{2(1 - \varepsilon)}{(T - T_0 )/2} + \lambda\varepsilon C_0,
\]
which, for a fixed positive $\epsilon$, is positive if $\lambda$ is
sufficiently large.
\end{proof}

We now consider $\widehat \beta$ as in the above lemma. The contactization of
$(W_{T_0}, \beta_1)$ embeds into the contactization of $(W_\infty, \widehat\beta)$.
Since the contactizations of $\beta_0$ and $\lambda\beta_0$ are isomorphic, 
the following lemma finishes the proof of Proposition~\ref{prop:embeddings}.

\begin{lemma}
The contactizations of $\widehat \beta$ and $\lambda\beta_0$ are isomorphic.
\end{lemma}

\begin{proof}
Consider the path of 1--forms
$\lambda_s=(1 - s)\lambda\beta_0 + s\widehat \beta$. The preceding lemma says
that $d\lambda_s$ tame $J$ for all $s$ and thus are all symplectic forms on
$W_\infty$. Thus $\alpha_s=dt+\lambda_s$ is a path of contact forms. 
Moser's technique provides an isotopy that connects the corresponding contact structures if the vector fields constructed in Moser's technique can be integrated for a sufficiently long time. It is clear that it can be so integrated since $\beta_0$ and $\widehat \beta$ coincide
outside the compact set $\overline{W_T}$.
\end{proof}

\subsection{Taming  \texorpdfstring{$J$}{J}}
\label{sec:gettingJ}

In this subsection we  prove Proposition~\ref{prop:Jtamed}. We begin with a general lemma that holds in any Riemannian manifold.
We will need the auxiliary functions
\begin{equation}\label{eq:H1-H2}
H_1(r) =  \sqrt{1 + \left(\frac{\sn_\kappa(r)}{r}\right)^2},\quad \text{ and } \quad
H_2(r) = {\textstyle \frac{4}{3}|\sec(g)|} \left(r H_1(r) + \int_0^r H_1(t)\, dt\right).
\end{equation}

\begin{lemma}\label{lem:XY}
Suppose the sectional curvature of $(M,g)$ is bounded below by some constant
$\kappa$.  Let $\gamma : [0, R] \to M$ be a unit speed geodesic with
$R < \injg$ and set $p = \gamma(0)$. Let $X(r)=(d\exp_p)_{r\gamma'(0)}v$ be
a vector field along $\gamma$ obtained as the image of a fixed unit vector $v\in
T_pM$. (We think of $v$ as a vector in $T_w(T_p(M))$ for every $w\in T_pM$ using
the canonical identification of a vector space with its tangent space at any
point.) 
Then, using notations from Equation~\eqref{eq:H1-H2} we have
\begin{align}
\label{eq:|XY|} &  \|X(r)\| \leq H_1(r)
\end{align}
and
\begin{equation}\label{eq:XY-nabla}
	\|\nabla_{\dot{\gamma}} X(r)\| \leq  H_2(r),
\end{equation}
for $r\in \bigl[0,\rmax\bigr)$ where $\rmax$ was defined in
Equation~\eqref{eq:r-interval}.
\end{lemma}

\begin{proof}
Let $J$ be the Jacobi field along $\gamma$ which satisfies $J(0) = 0$ and
$J'(0) = X(0)$. According to \cite[Theorem II.7.1 page 88]{Chavel06}, one has
$J(r) = rX(r)$. One can decompose $X$ into $X^\top$ 
which is parallel to $\dot\gamma$ and $X^\perp$ which is perpendicular to
$\dot\gamma$. Then by the Gauss lemma, one has $\|X^\top(r)\| = \|X^\top(0)\|$. The
perpendicular part is estimated by Rauch's theorem \cite[Theorem IX.2.3 page
390]{Chavel06} which gives $\|X^\perp(r)\| \leq \sn_\kappa(r)/r\, \|X^\perp(0)\|$.
Thus we have
\begin{align*}
	\|X(r)\|^2 &= \|X^\top(r)\|^2 + \|X^\perp(r)\|^2 \\
	&\leq 1 + \frac{\sn_\kappa(r)^2}{r^2}.
\end{align*}
We will now establish Equation~\eqref{eq:XY-nabla}. Since $X = r^{-1}J$, we have
$X' = r^{-1}(J' - X)$. The Jacobi equation for $J$
reads $J'' + RJ =0$ where $RJ$ is short hand for $R(J,\gamma')\gamma'$. 
In a parallel frame along $\gamma$, the components of this equation can be
integrated component-wise.  We can now estimate
\begin{align*}
\|X'(r)\| &= r^{-1} \|J'(r) - X(r)\| = r^{-1}\left\|\int_0^r -RJ - X'\right\|  \leq r^{-1}\int_0^r \|RJ\| + \|X'\|  \\
&\leq \frac{4}{3}|\sec(g)|\, \max \|J\| + r^{-1}\int_0^r \|X'\|\\
&\leq \frac{4}{3}|\sec(g)|\, r H_1(r) + \frac 1r\int_0^r \|X'\|.
\end{align*}
where the second inequality follows from \cite[p. 95]{Chavel06} just as in the proof of Lemma~\ref{lemma:nDJprime}. So estimating $\| X'\|$ is now a Gr\"onwall type problem. We set $f(r) = \|X'(r)\|$
and $\alpha(r) = \frac{4}{3}|\sec(g)|\, r H_1(r)$, so that the above inequality reads
\begin{equation}
	\label{eq:gronwall}
f(r) \leq \alpha(r) + \frac 1r \int_0^r f(t)\, dt.
\end{equation}
Setting $v(r) = \frac 1r \int_0^r f(t)\, dt$ and keeping in mind that $f$ is smooth and $f(0)=0$, we have
\[
v'(r) = \frac 1r \left(f - \frac 1r\int_0^r f(t)\,dt\right) \leq \alpha(r)/r.
\]
Since $v(0) = 0$, we see that $v(r) \leq \int_0^r \frac{\alpha(u)}u \,du$, which can be 
substituted into Equation~\eqref{eq:gronwall} to obtain the announced estimate.
\end{proof}

\begin{lemma}\label{lem:estimateFF'}
Given a radial geodesic $\gamma$ in $\D(r_0)$ starting at
$p=\gamma(0)$, consider two vector fields $X$ and $Y$ along $\gamma=\gamma(r)$ 
which are the images of unit vectors in $\xi_p$ under the differential of
$\exp_p$ as described in Lemma~\ref{lem:XY}.
Then the derivative of the function $F(r) = (1/\theta')\,d\alpha(X, Y)$ is
bounded by the constant $\overline{H}$ defined in Equation~\eqref{eq:H1-H2-bar-consts},
i.e.
\begin{align}
 \label{eq:F'} |F'(r) | & \leq \overline H,\qquad \text{for}\quad r\in [0,\rmax],
\end{align}
where $\rmax$ is defined in Equation~\eqref{eq:r-interval}.
\end{lemma}
\begin{proof}
By Proposition \ref{proposition:determinemetric} the metric $g$ can be expressed as 
\[
 g(X,Y)=\frac{1}{\theta'} d\alpha(X,\phi(Y))+\alpha(X)\alpha(Y),
\]
therefore
\[
	\frac{1}{\theta'} d\alpha(X, Y)=-\g{X,\phi(Y)}
\]
and
\[
	\frac{d}{dr} F(r) =
	-\g{\nabla_r X,\phi(Y)} - \g{X,\phi(\nabla_r Y)} - \g{X,(\nabla_r \phi)(Y)}.
\]
So Lemma~\ref{lem:XY} and estimate~\eqref{eq:nabla-phi-norm} yield
\begin{align*}
\Bigl|\frac{d}{dr} F(r)\Bigr|  &\leq 
H_2 \|\phi(Y)\| + H_2 \|X\| + (\|[J, J]\|/2 + 2B)\|X\|\cdot \|Y\|\\
&\leq 2H_1 H_2 + (\|[J, J]\|/2 + 2B)H_1^2.
\end{align*}

Note that $H_1(r)$ is increasing when $\kappa<0$, constant (equal $\sqrt{2}$) for $\kappa=0$, and
decreasing on $[0, \frac{\pi}{2\sqrt{\kappa}}]$  to $1$ at $\frac{\pi}{2\sqrt{\kappa}}$ if $\kappa> 0$.  Also the function $H_2(r)$ is increasing and vanishing at $r=0$. Thus, on the interval $[0,\rmax)$ where $H_1$ and $H_2$ are defined we have the following simple estimates
\begin{equation}\label{eq:H1-H2-const}
\begin{split}
 H_1(r)  &\leq \overline{H}_1:=\max_{[0,\rmax)} H_1(r)=\begin{cases}
  H_1(0),&  \text{if}\ \kappa\geq 0,\\
  H_1(\rmax) & \text{if}\  \kappa<0.
\end{cases}\\
H_2(r) & \leq \overline{H}_2=\frac{8}{3}|\sec(g)|\overline{H}_1 \rmax.
\end{split}
\end{equation}
In addition both $H_1$ and $H_2$ are non-negative functions hence
$|F'(r)|\leq \overline{H}$.
\end{proof}

\begin{proof}[Proof of Proposition~\ref{prop:Jtamed}]
Given any nonzero vector $u$ tangent to $D(r)$ as some point $q$ let $\gamma$ be the geodesic from $p$ to $q$. 
Choose a symplectic basis $X_1, \dots, X_n, Y_1, \dots, Y_n$ of $\xi$ at $p$
with $Y_i = JX_i$ and $|X_i| = 1$ and use the same notation for the vector fields along $\gamma$ 
obtained using $\exp_p$ as in Lemma~\ref{lem:XY}. This gives a moving
frame for $D(r)$. Hence there are constants $a_i$ and $b_i$ such that
$u = \sum_i a_i X_i + b_i Y_i$.
Then
\begin{align*}
d\alpha(u, Ju) &= \sum_{1\leq i,j\leq n} 
	\Bigl(-a_i b_j d\alpha(X_i,X_j) 
	+a_i a_j d\alpha(X_i,Y_j) 
	-b_i b_j d\alpha(Y_i,X_j) 
	+a_j b_i d\alpha(Y_i,Y_j)\Bigr) \\
&= \sum^n_{i=1} (a_i^2 + b_i^2) d\alpha(X_i, Y_i) + \sum_{1\leq i\neq j\leq n} 
	\Bigl(-a_i b_j d\alpha(X_i,X_j) 
	+a_i a_j d\alpha(X_i,Y_j) \\
&\hspace{7cm} 	-b_i b_j d\alpha(Y_i,X_j) 
	+a_j b_i d\alpha(Y_i,Y_j)\Bigr).
\end{align*}
At  $r = 0$, $(1/\theta')\, d\alpha(X_i, Y_i) = 1$, and $d\alpha(Z, W) = 0$ if $Z$ and $W$ are any pair of vectors appearing in the 
second sum above. 
From the derivative estimate in Equation~\eqref{eq:F'} we see that
\[
\frac{1}{\theta'} d\alpha(X_i, Y_i) \geq 1 - \overline{H} r \qquad \text{and} \qquad
\frac{1}{\theta'} |d\alpha(Z, W)|(r) \leq \overline{H} r.
\]
To continue the computation, extend the $X_i$, $Y_i$ and $\nD$ to a neighborhood of $\gamma$ and use them to define an auxiliary metric on the neighborhood so that they from an orthonormal basis. In this metric the norm squared of $u$, which we denote by $N(u)$, is $\sum_{i=1}^n (a_i^2+b_i^2)$. We now have
\begin{align*}
	\frac{1}{\theta'} d\alpha(u, Ju) &\geq (1 - \overline{H} r)N(u) - r\overline{H}\sum_{1\leq i\neq j\leq n} 
	\Bigl(|a_i| |b_j| +|a_i| |a_j|+|b_i| |b_j| +|a_j| |b_i|\Bigr)
\end{align*}
Since there are $2\,\binom{n}{2}=n(n-1)$ 
terms in the sum and $\frac 12 N(u)\geq |a_\ast||b_\ast|$ (or $|a_\ast||a_\ast|$ or $|b_\ast||b_\ast|$), we see that
\[
	\frac{1}{\theta'} d\alpha(u, Ju)  \geq \left(1 - (1+2n(n-1))\overline{H}r\right)N(u).
\]
Hence, for $u\not= 0$ we have $d\alpha(u, Ju)>0$ if $r<\left((1+2n(n-1))\overline{H}\right)^{-1}$.
\end{proof}

\subsection{From cylinders to balls}
\label{sec:cyl2ball}

The goal of this subsection is to prove Proposition~\ref{prop:cyl2balls} which
guarantees that Reeb flow cylinders contain embedded geodesic balls of a certain radius.

\begin{proof}[Proof of Proposition~\ref{prop:cyl2balls}]
Recalling the hypothesis of the proposition we have a complete Riemannian
manifold $(M,g)$ whose sectional curvatures are bounded above by $K$ and a
number  $r_0$ such that
\[
r_0 < \min\left(\frac{\injg}2, \frac{\pi}{2\sqrt{K}}\right).
\]
The vector field $X$ is a unit speed geodesic vector field on $M$ and
$p$ is some point in $M$. Let $\xi$ the hyperplane field $X_p^\perp$. We
set $D_{\xi}(r_0)=\{v\in X_p^\perp, \|v\|<r_0\}$ and
\[
\D(r_0)=\{\exp_p(v) :\; v\in D_{\xi}(r_0) \}.
\]
Moreover, we set
\[
E:\R\times D_{X_p}(r_0)\to M: (t,v)\mapsto \Phi(t, \exp_p(v)),
\]
where $\Phi$ is the flow of $X$. We finally set $\C(r_0)=\text{image} (E)$. 

Denote by $\mathsf{B}^\text{\rm conv}$ the ball at $p$ of radius  $r_0$,
and $\C^\text{\rm conv}(r_0)=\C(r_0)\cap \mathsf{B}^\text{\rm conv}$ the portion
of $\C(r_0)$ in $\mathsf{B}^\text{\rm conv}$. Note that
Proposition~\ref{prop:riem_convexity} implies that the sphere 
$\partial \mathsf{B}^\text{\rm conv}$ is convex.
Let $B$ be the connected component of $E^{-1}(\C^\text{\rm
conv}(r_0))$ that contains $(0,0)$. Since $r_0<\text{\rm conv}(g)$ we
observe that $E$ restricted to $B$ is an embedding. To see this
we need to study the structure of $B$. We first notice that $B\cap \xi_p$ is 
$D_\xi(r_0)$, the ball of radius $r_0$ in $\xi_p$. Now for each $x\in D_\xi(r_0)$ let 
$I_x=(\R \times \{ x\})\cap B$. So $B=\cup_{x\in D_\xi(r_0)} I_x$. Each $I_x$ clearly contains 
the origin ({\em i.e.\ }$(0,x)$.) We moreover claim that it
is a connected interval for all $x$. If not, then there is some $x_0$ for which it is
not connected. Let $\gamma$ be the geodesic in $\D(r_0)$ 
from $p$ to $x_0$. Let $J_x$ be the smallest interval containing $I_x$ for each
$x\in \gamma$ and $K_x$ the image of $x$ under the flow of $X$ for times in $J_x$. 
Finally the union of all $K_x$ for $x\in \gamma$ is a disk $A$. By hypothesis $A$
is not contained in $\mathsf{B}^\text{\rm conv}$. If $x'$ is the first point on $\gamma$ 
such that $K_{x'}$ is not contained in the interior of $\mathsf{B}^\text{\rm
conv}$, then $K_{x'}$ is a geodesic which has an interior tangency with
$\partial \mathsf{B}^\text{\rm conv}$. This contradicts the convexity of
$\mathsf{B}^\text{\rm conv}$.  Thus the $I_x$ are all connected.  This implies
that the restriction of $E$ to $B$ is injective because, since we are below the
injectivity radius, the ball $\mathsf{B}^\text{\rm conv}$ is divided into two
connected components by $\D(r_0)$ and, in particular, there is no trajectory of
$X$ leaving $\D(r_0)$ and returning to it without leaving 
$\mathsf{B}^\text{\rm conv}$.
 
Next, we estimate a  radius $\bar{r}$ of a convex ball centered at $p$ and
contained in $\C^\text{\rm conv}(r_0)$, a maximal such ball $B_p(\bar{r})$ will
either be of radius $\geq r_0$ or have a tangency with $(\partial \C(r_0))\cap
\mathsf{B}^\text{\rm conv}$. Let $q$ be a point of the tangency and $v$ be
the point of intersection of the orbit of $X$ through $q$  and the 
disk $\D(r_0)$. Clearly $v\in (\partial \C(r_0))\cap \mathsf{B}^\text{\rm conv}$ and
thus $d(p,v)=r_0$. Moreover we know that $d(p,q)=\bar{r}$.
Consider the geodesic triangle $T(p,q,v)$ consisting of the unique geodesics connecting
these points. So both sides $(p,q)$ and $(p,v)$ of $T(p,q,v)$ are radial geodesics
emanating from $p$, and $(q,v)$ is a piece of the orbit of $X$. This
piece of orbit contained in $\mathsf{B}^\text{\rm conv}$ otherwise it
would contradict the convexity of $\partial\mathsf{B}^\text{\rm conv}$.
So $d(q, v) \le \mathrm{diam}(\mathsf{B}^\text{\rm conv}) = 2r_0$ and the
perimeter of $T(p,q,v)$ is less than $4r_0$.

We choose
$\gamma(s)$ to be the radial geodesic parameterizing $(p,v)$, {\em i.e.\ }$\gamma(0)=p$
and $\gamma(r_0)=v$.  Denote by $0\leq \phi\leq \frac{\pi}{2}$ the angle
between $X(\gamma(r_0))$ and the normal $n(\gamma(r_0))$, and by 
$0\leq \varphi \leq \frac \pi 2$ the angle
between $X(\gamma(r_0))$ and the line spanned by $\gamma'(r_0)$. Observe that
$\frac{\pi}{2}-\phi \leq \varphi$, thus using
Estimate~\eqref{eq:angle-est} we have
\begin{equation}\label{eq:trig1} 
\sin(\varphi)
\geq
\cos(\phi)
\geq
1-P(r_0).  
\end{equation}
Toponogov's theorem (in a slightly unusual curvature upper bound setup,
see e.g. \cite[Theorem 4.1 p.~197]{Karcher89}) allows to compare with
a geodesic triangle $T=T(p',q',v')$ in the space form
$M_K$, based at vertex $v'$ with angle $\varphi'=\varphi$ at $v'$, 
$d(p',v')=d(p,v)=r_0$, and $d(q',v')=d(q,v)$. 
The sectional
curvature is bounded above by $K$ and the perimeter of $T(p, q, v)$ is
less than $\min(2\injg, 2\pi/\sqrt K)$ so Toponogov's theorem gives
\[ 
d(p',q')\leq d(p,q)=\bar{r}.
\]
Now, let $\theta'$ be the angle at the vertex $q'$ in the reference triangle.
The law of sines \cite[Note II.5 page 103 and references therein]{Chavel06}
applied to the triangle $T(p',q',v')$  yields
\[
\frac{\sin{\theta'}}{\sn_K(r_0)} = \frac{\sin{\varphi}}{\sn_K(d(p', q'))}
\]
This combines with $\sin(\theta') \leq 1$ and Inequality~\eqref{eq:trig1} to
give
\begin{equation*}
\sn_K(\bar{r}) \geq \sn_K(r_0)(1 - P(r_0)).
\end{equation*} 
hence the announced bound since the functions $\sn_K$ are increasing.
\end{proof}

\section{Size of standard neighborhoods in dimension 3}
\label{S:geometric_method}

In this section we show how to use a geometric method similar to the one
in Section~\ref{sec:darboux} to provide a stronger estimate on the
Darboux radius in dimension 3. We also show how to use this idea to
construct standard neighborhoods of closed geodesics with an estimated
size. 

In Subsection~\ref{SS:geom_to_topo} we state two propositions that essentially
say if $N=\zeta\times D^2$ is a neighborhood of a Reeb orbit and one can
control the twisting of the contact structure or its Reeb vector field then the
contact structure on $N$ can be embedded in the standard contact $\R^3$. We
then prove our main results in dimension 3 using previously derived geometric
estimates of Section~\ref{sec:darboux}. In the following sections we then prove
facts about characteristic foliations that were needed in the proofs of our
main result. 

\subsection{From geometric control to topology}
\label{SS:geom_to_topo}

The geometric setup used throughout this section is the following. We
consider $(\alpha, g, J)$ a contact metric structure compatible with
$(M,\xi)$, and denote by $R_\alpha$ its Reeb vector field. Let $\zeta$
be a portion of a geodesic which is either an arc or a circle.
We denote by $\nu\zeta$ the normal bundle to $\zeta$ with
respect to $g$. We denote by $\exp_\nu : \nu\zeta \to M$ the restriction
to $\nu\zeta$ of $\exp : TM \to M$ (we assume as usual that $g$ is
complete).

We fix a positive radius $r$ less than the injectivity radius $\injg$
and consider, for each $z$ in $\zeta$\, the embedded disk
\[ 
	\D_z := \exp_\nu\left(\{ v \in \nu_z \zeta \;;\;
|v| \leq r\}\right) 
\] and we denote by $\nD$ the unit vector field
orthogonal to all $\D_z$ which coincides with $R_\alpha$ along $\zeta$.

For each radius $r$ we denote by 
\[ 
	\T(r) = \exp_\nu\left(\{ v \in \nu \zeta \;;\; 
	|v| \leq r\}\right) 
\] 
the tube of radius $r$ around $\zeta$. 
It is the image of either a solid torus or a thickened disk depending on
whether $\zeta$ is a circle or an arc but $\exp_\nu$ may fail to be an
embedding. The following lemma, which has nothing to do with contact
geometry and will be proved at the end of this section, gives a
sufficient condition to have an embedding in the thickened disk case.

\begin{lemma}
\label{lemma:Tembedded}
Suppose the sectional curvature of $g$ is bounded above by $K$ and
let $r$ be a positive number such that
\[
	r < \min\left(\frac{\injg}2, \frac{\pi}{2\sqrt{K}}\right).
\]
For every geodesic arc $\zeta$ with length less than $2r$, 
the tube $\T(r)$ is embedded.
\end{lemma}

Returning to the specific setting of contact geometry, we will first
need the following result.
\begin{proposition}
\label{prop:geomtop_closed}
If $\zeta$ is a closed Reeb orbit such that $\T(r)$ is embedded and $\xi$ is
transverse to $\nD$ then $(\T(r), \xi)$ is contactomorphic to a domain in
$(S^1\times \R^2, \ker(d\phi+\rho^2\, d\theta)),$ where $\phi$ is the angular
coordinate on $S^1$ and $(\rho,\theta)$ are polar coordinates on $\R^2.$ In
particular, $(\T(r),\xi)$ is universally tight.
\end{proposition}
In the case where $\zeta$ is an interval in a Reeb orbit, we will use
the embedding criterion above and impose the
stronger condition that the Reeb vector field $R_\alpha$ stays
transverse to all $\D_z$. It will allow us to embed $\T(r)$ in a solid
torus which we will then embed in the standard $(\R^3, \xistd)$.
\begin{proposition}
\label{prop:geomtop_open}
In the above setup and under the hypothesis of
Lemma~\ref{lemma:Tembedded}, if $\zeta$ is an arc and the Reeb vector
field $R_\alpha$ is transverse to the disks $\D_z$, then 
$(\T(r),\xi)$ is contactomorphic to a domain in $(\R^3,\xistd).$
\end{proposition}
Now, Theorems \ref{thm:geometric} and \ref{thm:Reeb_nbhds} follow at once.

\begin{proof}[Proof of Theorem \ref{thm:geometric}]
By the hypothesis of the theorem and Propositions~\ref{prop:t-estimate} 
and~\ref{prop:geomtop_open} the
cylinder $\T(r)$ is embedded and the contact structure on $\T(r)$ is
contactomorphic to a domain in $(\R^3,\xi_{std})$. Thus any geodesic ball
embedded in $\T(r)$ is
standard. We will now prove that the geodesic ball $B(p, r)$ of radius $r$ about $p$ 
is contained in $\T(r)$ and thus the theorem will follow. We denote the boundary of
$B(p, r)$ by $S(p,r)$. 

Since $\T(r)$ contains very small balls around $p$, there is a positive first
radius $r_0$ such that the sphere $S(p, r_0)$ intersects $\partial\T(r)$.  We
denote by $p_\pm = \zeta(\pm r)$ the two extremities of $\zeta$.  We denote by
$D_\pm$ the geodesic disks normal to $\zeta$ centered at $p_\pm$.  The boundary
of $\T(r)$ is made of $D_\pm$ which we will call horizontal together with the
vertical part made of points at distance $r$ from $\zeta$. One can see that the
intersection between $S(p, r_0)$ and $\partial\T(r)$ cannot meet the intersection
between the horizontal and vertical parts. At a point in the interior of the
vertical part, the plane tangent to $\partial\T(r)$ is the orthogonal to a
geodesic which is normal to $\zeta$ (by the generalized Gauss Lemma see
e.g. \cite[Lemma 2.11 page 26]{Gray}). Likewise,
the tangent space of $S(p, r_0)$ at any point is the orthogonal to a geodesic
normal to $S(p, r_0)$ and going through $p$. So, at a point where $S(p,r_0)$ is
tangent to the vertical part of $\partial\T(r)$, those geodesics coincide and we
see that the relevant point of $\zeta$ is $p$.  Hence $r_0 = r$ in this case. 

Suppose now that the intersection between $S(p, r_0)$ and $\partial\T(r)$ is in
$D_+$ or $D_-$.  Let $q$ be a tangency point in one of those disks, say $D_+$.
Let $\gamma$ be the geodesic from $p_+$ to $q$ (recall we are below the
convexity radius). Note that $\gamma$ is normal to $\zeta$. Let $\gamma'$ be the
geodesic between $p$ and $q$.  Since $S(p,r_0)$ is tangent to $D_+$ at $q$
the Gauss lemma implies that $\gamma'$ is normal to $\gamma\subset D_+$. 
Now there are two cases to consider. If $q = p_+$ then $r_0 = r$ and we have
established our claim.  If $q\not= p_+$ then we
contradict Lemma~\ref{lemma:Tembedded}, applied to $\gamma$, because the tube around
$\gamma$ with radius $r$ is not embedded. Indeed we have two geodesics $\zeta$
and $\gamma'$ normal to $\gamma$ with length at most $r$ which intersect.
\end{proof}

\begin{proof}[Proof of Theorem~\ref{thm:Reeb_nbhds}]
By the hypothesis of the theorem and Proposition~\ref{prop:t-estimate}, the geodesic tube 
$\T(r)$ about the geodesic $\gamma$ is embedded
and the Reeb vector field stays transverse to all the disks $\D_z$. Since the
contact planes are perpendicular to the Reeb vector field, this implies that
they are transverse to the vector field $\nD$ normal to the disks. The theorem
now follows from Proposition \ref{prop:geomtop_closed}.
\end{proof}

The next sections are devoted to the proof of
Propositions~\ref{prop:geomtop_closed} and~\ref{prop:geomtop_open} but first we
come back to Lemma~\ref{lemma:Tembedded}.

\begin{proof}[Proof of Lemma~\ref{lemma:Tembedded}] 
By the assumption  $r$ is less than the injectivity radius.
Thus, the restriction of $\exp_\nu$ to $\{z\}\times D(r)$ for every fixed $z$ is
injective. Thus injectivity of $\exp_\nu$ can fail only when there is a
geodesic triangle $T(p,q,v)$ in $\T(r)$ formed by a piece
of the geodesic $\zeta$ between $p=\zeta(z_0)$ and $q=\zeta(z_1)$ in
$\T(r)$ and two ``radial'' geodesics $\gamma_0$ in $\D_{z_0}$,
and $\gamma_1$ in $\D_{z_1}$, such that $\gamma_0(0)=p$,
$\gamma_1(0)=q$ and $\gamma_0(t_0)=\gamma_1(t_1)=v$ for some $t_0, t_1\leq
r$. 
Since $\zeta$ is orthogonal to
$\D_{z_0}$ and $\D_{z_1}$, the triangle $T$ also has two right
angles at vertices $p$ and $q$ and its perimeter is at most $4r$ which
is less than $\min(2\injg, 2\pi/\sqrt K)$. Toponogov's theorem
\cite[Theorem 4.1 p.~197]{Karcher89} first guaranties there exists an
analogous triangle in the model space $M_K$. So we already get a
contradiction is $K$ is non-positive since there is no triangle with two
right angles in Euclidean or hyperbolic space. We now assume that $K$ is
positive.

Without loss of generality we assume
\begin{equation}\label{eq:d(q,v)<=d(p,v)}
d(q,v)\leq d(p,v),
\end{equation}
otherwise we may switch the labels of the vertices $p$ and $q$ (the only
feature of the vertex needed is that the angle at that vertex is $\frac \pi 2$).
Now, compare the right
geodesic triangle $T(p,q,v)$ to the right
triangle $T'(p', q', v')$,
in the $2$-sphere of constant curvature $K$. Since at $p'$ the
angle between the sides of the triangle $T'$ is
$\frac\pi 2$ and $d(p,q)=d'(p',q')$, $d(p,v)=d'(p',v')$, the spherical law of cosines radius implies,
\[
d'(p',v') < d'(q',v').
\]
However, this yields a contradiction as
\[
d(p,v)=d'(p',v')< d'(q',v')\leq d(q,v)\leq d(p,v),
\]
where the second inequality is a consequence of Toponogov's theorem
cited above, and the last inequality is Equation~\eqref{eq:d(q,v)<=d(p,v)}. 
\end{proof}

\subsection{Background on characteristic foliations}
Recall that an oriented singular (this adjective will be implicit in the
following) foliation on an oriented surface $S$ is an equivalence class
of 1--forms where $\alpha \sim \beta$ if there is a positive function $f$ such
that $\alpha = f\beta$. 
Let $\alpha$ be a representative for a singular
foliation $\F$. 
A singularity of $\F$ is a point where $\alpha$ vanishes. The
singularity $p$ is said to have non-zero divergence if
$(d\alpha)_p$ is an area form on $T_pS.$ If $\omega$ is
an area form on $S$ (compatible with the chosen
orientation) then to 
each singular point $p$ we attach the sign of the unique real number $\mu$ such that 
$(d\alpha)_p = \mu \omega_p$. One can easily check that singular points and their signs
do not depend on the choice of $\alpha$ in its equivalence class or of $\omega$ if
we keep the same orientation.

Let $S$ be an oriented surface in a contact manifold $(M, \xi)$ with $\xi = \ker
\alpha,$ co-oriented by $\alpha$. The characteristic foliation $\xi S$ of $S$ is
the equivalence class of the restriction of $\alpha$ to $S$. The contact
condition ensures that all singularities of characteristic foliations have
non-zero divergence and hence have non-zero sign. Singularities of $\xi S$
correspond to points where $S$ is tangent to $\xi$ and they are positive or
negative according as the orientation of $\xi$ and $S$ match or do not match.
We also notice that $\alpha$ provides a co-orientation, and hence if $S$ is 
oriented by an area form $\omega$ the orientation of the line field $\xi S$ is given by 
the vector field $X$ which satisfies $\iota_X \omega=\alpha|_S$.
One may dually think of the characteristic foliation on $S$ as coming from the
singular line field on $S$ given by $T_pS\cap \xi_p$ for each $p\in S.$

\subsection{Characteristic foliations on tori and contact embeddings}

We will need to show, informally speaking, how a contact structure which is
transverse to the core of the solid torus and ``does not rotate more than half a
turn between the core and the boundary'' embeds inside the standard contact
structure on  $\R^3$. In the following we make this statement precise and provide 
a proof of it. We will denote by $T$ a torus and by $T_t$ the torus $T \times \{t\}$ in the
thickened torus $T \times [0, 1]$.

We first recall some notions about suspensions on tori.  A non-singular
foliation $\mathcal{F}$ on $T$ is called a suspension if there is a simple
closed curve intersecting all leaves transversely. The name comes from the fact
that $\mathcal{F}$ can be reconstructed by suspending the Poincar\'e first
return map on the transversal curve. To such a foliation one can associate a
line in $H_1(T; \R)$. This line $d(\mathcal{F})$ is called the asymptotic
direction of $\mathcal{F}$. We briefly sketch the construction. Pick any point
$x$ in $T$, follow the leaf of $\mathcal{F}$ through $x$ for a length $T$, and
create a closed curve $O(x, T)$ using a geodesic (for some auxiliary metric).
Then the limit homology class $\lim_{T \to \infty} \frac{1}{T} [ O(x, T) ]$
exists for every $x$ and it defines a line in $H_1(T; \R)$ that does not depend
on $x$ or $T$. The limit is called the {\em asymptotic direction}. Two easy
examples are when $\mathcal{F}$ is linear (we recover the intuitive notion of
direction) and when there is a closed leaf (its asymptotic direction is spanned
by the homology class of this leaf).

Let $T \times [0, 1]$ be a thickened torus.  If the characteristic foliations on
all the tori $T_t$ induced by some contact structure $\xi$ are suspensions then
the contact condition forces the asymptotic directions $d(\xi T_t)$ to always
rotate continuously in the same direction (which is determined by the
orientations of the manifold, the contact structure and the tori). This
direction can be constant along some sub-intervals\footnote{This obviously
happens around each $t$ such that $\xi T_t$ is structurally stable.} but it
cannot be constant in a neighborhood of $t$ if $\xi T_t$ is linear.

If a contact structure $\xi$ on a solid torus $W$ is transverse to a core curve
$K$ of $W$ then it lifts to a contact structure on the toric annulus $T\times
[0, 1]$ obtained by blowing%
\footnote{The knot $K$ has a tubular neighborhood 
$D^2_\varepsilon \times S^1$ with
coordinates $(re^{i\theta}, \varphi)$ such that 
$\xi = \ker d\varphi + r^2d\theta$. The blow up map from 
$[0, 1] \times S^1 \times S^1$ to $D^2 \times S^1$ is simply
$(s, \theta, \varphi) \mapsto (\sqrt{s}e^{i\theta}, \varphi)$ on $[0, \varepsilon^2] \times S^1 \times S^1$ and interpolated to $(s, \theta, \varphi) \mapsto (s e^{i\theta}, \varphi)$ on $[\varepsilon, 1] \times S^1 \times S^1$.
}
 up $K$. The lifted contact structure induces a
linear foliation on the boundary component which projects to $K$, say $T_0$. The
direction of this foliation is spanned by the meridian class, i.e. the class in
$H_1(T)$ which spans the kernel of the map from $H_1(T \times [0, 1])$ to
$H_1(W)$ induced by the projection.

The following lemma gives a precise formulation of the idea described informally
at the beginning of this subsection.
\begin{lemma}\label{lem:ct}
Let $\xi'$ be a contact structure on a solid torus $W$ transverse to a core curve 
$K$ of $W$. Let $\xi$ be the  contact structure on $T \times [0, 1]$ lifted from $W$ 
as described 
above. If all characteristic foliations $\xi T_t$ are suspensions whose
asymptotic directions never contain the meridian homology class for $t > 0$
then $(W, \xi')$ is contactomorphic to a  domain in
$(S^1\times \R^2, \xi_{rot}=\ker(d\phi+r^2\, d\theta)).$ 
\end{lemma}
This lemma is an easy consequence of the following result of Giroux. 
\begin{theorem}[{Giroux 2000, \cite[Theorem 3.3]{Giroux00}}]
Suppose two contact structures $\xi_0$ and $\xi_1$ on $T\times[0,1]$ induce
suspensions on each torus $T_t$ which agree on $T_0$ and $T_1$. If the two paths
of asymptotic directions $t \mapsto d(\xi_i T_t)$ are non constant and homotopic
relative to their common end-points then $\xi_0$ and $\xi_1$ are isotopic
relative to the boundary $T_0 \cup T_1$.
 \qed
\end{theorem}

\begin{proof}[Proof of Lemma~\ref{lem:ct}]
Because of the asymptotic direction assumption, we can choose a longitudinal
curve $L$ for $W$ whose homology class does not belong to the asymptotic
direction of any of the foliations $\xi T_t$ (pick any longitude and add a
sufficiently large multiple of the meridian). In particular we can choose $L$ to
intersect transversely all leaves of $\xi \partial W$ and such that the
corresponding Poincar\'e map rotates all points clockwise. So there is an
identification of $(W, \partial W)$ with $(D^2 \times S^1, S^1 \times S^1)$
sending $L$ to some $\{*\} \times S^1$ and $\xi \partial W$ to a foliation
directed by $v=\frac{\partial}{\partial \theta} + F(\theta,\phi)
\frac{\partial}{\partial \phi}$ where we use coordinates $(\theta,\phi)$ on
$S^1\times S^1$ and $F(\theta,\phi)\leq 0$.

We now consider an embedding $\varphi$ of $D^2 \times S^1$ into $\R^2 \times
S^1$ which is the identity along $\{0\} \times S^1$ and agrees with $(1,
\theta, \phi) \mapsto ((-F(\theta,\phi))^{\frac 12},\theta, \phi)$ along $S^1
\times S^1$. An immediate computation reveals that $\varphi$ sends $\xi \partial
W$ to $\xi_{rot} \varphi(\partial W)$. Using the standard neighborhood theorem
for curves transverse to a contact structure, it can also be easily arranged
that $\varphi$ sends $\xi$ to $\xi_{rot}$ along $K$.  One may now blow up $K$
and apply Giroux's theorem to further isotope $\varphi$ to a contact
embedding. The homotopy hypothesis is guaranteed because the homology class of
$[L]$ belongs to no asymptotic direction of a $\xi T_t$.
\end{proof}

\begin{proof}[Proof of Proposition~\ref{prop:geomtop_closed}]
The vector field $\nD$ is tangent to each torus $T_r = \partial \T(r)$, $0 < r
\leq r_0$ and $\xi$ is transverse to $\nD$ so each characteristic foliation $\xi
T_r$ is non-singular and transverse to $\nD$. In addition, $\nD$ is transverse to the foliation
by meridian circles of $T_r$ coming from the disks $D_z$ so it is a
suspension. So there is a simple closed curve $\mathcal C$ in $T_r$ which is not
a meridian and transversely intersects all leaves of $\xi T_r$ (the homology
class of $\mathcal C$ belongs to a rational approximation of the asymptotic
cycle of $\nD$). For brevity, we denote the asymptotic directions $d(\xi T_r)$ and 
$d(\nD)$ by $X$ and $N$, respectively, and by $M$ the line spanned by the meridian
homology class. One can see $X$ and $N$ as points on the circle $P(H_1(T^2;\R))$
continuously moving as $r$ increases. Note that $X$ initially equals $M$ and
moves clockwise in a monotone way whereas $N$ moves in some way but never hits
$M$ or $X$. This easily implies that $X$ cannot become meridional and one can
apply Lemma~\ref{lem:ct} to obtain the desired conclusion.
\end{proof}

In order to prove Proposition~\ref{prop:geomtop_open}, we first need to
understand neighborhoods of disks transverse to Reeb vector fields.
\begin{lemma}
Let $M$ be a 3--manifold with a contact form $\alpha$ and $D$ a closed disk
embedded in $M$. If the Reeb vector field of $\alpha$ is transverse to $D$ then
there is a neighborhood $V$ of $D$ in $M$ and an embedding $\varphi$ of $V$ in
$\R^3$ such that $\varphi(D)$ is transverse to vertical lines in $\R^3$ and
$\varphi_*\alpha = dz + r^2 d\theta$.
\end{lemma}
\begin{proof}
Let $D'$ be a disk containing $D$ and still transverse to the Reeb vector field
$R$. Let $\varepsilon$ be a positive number such that the flow of $R_\alpha$
embeds $D' \times [-\varepsilon, \varepsilon]$ into $M$. If $r$ and $\theta$ are
polar coordinates on $D'$ and $z$ is the coordinate in $[-\varepsilon,
\varepsilon]$ the pull-back of $\alpha$ is $dz + f(r,\theta)\,dr + g(r,
\theta)\, d\theta$ for some functions $f$ and $g$.  This contact form extends
trivially to $D' \times \R$ as the contactization of the exact symplectic
manifold $(D', \beta)$ where $\beta = f\, dr + g\, d\theta$. We now think of
$D'$ as the unit disk in $\R^2$. The following claim will be proved below.

\textbf{Claim:} There is an extension of $\beta$ to $\R^2$ such that $d\beta$
is symplectic everywhere and $\beta = r^2\, d\theta$ outside some large disk.

Since the symplectic condition is convex in dimension 2, $\alpha_t=dz+(1-t)\beta
+ t r^2\, d\theta, t\in [0,1],$ is a family of contact forms on $\R^2 \times
\R$.
We now use a general fact: if $\beta_t$ is a family of 1--forms on a surface $S$
such that each $d\beta_t$ is symplectic and all $\beta_t$ agree outside some
compact set, then there is an isotopy $\varphi_t$ of the contactization $S \times
\R$ sending surfaces $S \times\{z\}$ to surfaces transverse to lines $\{s\}
\times \R$ and such that $\varphi_t^*\alpha_t = \alpha_0$. The isotopy is
constructed using Moser's technique as the flow of a vector field $X_t = Y_t +
\lambda_t \partial_z$ with $Y_t$ the vector field on the surface defined by
$d\beta_t(Y_t, \cdot) = -\dot\beta_t$ and $\lambda_t = -\beta_t(Y_t)$. The
transversality condition comes from the commutation of $X_t$ and $\partial_z$.

We now prove the claim. First we fix the symplectic form $\omega = r\,
dr\wedge d\theta$ on $\R^2$ so that the problem is reformulated in terms of vector
fields $\omega$-dual to the forms we consider. We have a vector field $Y$ with
positive divergence on $D'$ and we want to extend it to $\R^2$ such that it
coincides with $Y_0 = r\partial_r$ outside some large disk. We denote by $D_r$
the disk of radius $r$ around the origin so $D' = D_1$. We first extend $Y$
arbitrarily to a neighborhood of $D_1$ so its divergence stays positive in some
$D'' = D_{1+\delta}$. Let $h$ be a smooth function from $[0, \infty)$ to $\R$
with $h$ and $h'$ vanishing on $[0, 1]$ and everywhere non negative. If $h$
grows sufficiently fast between $1$ and $1+\delta$, the vector field $Y' = Y +
h\partial_r$ has positive divergence and is transverse to the boundary of
$D''$. It is then easy to extend $Y'|_{D''}$ to a vector field which is
transverse to all circles $\partial D_r$ for $r \geq 1 + \delta$ and has the
same orbits as $Y_0$ outside the disk $D_{1 + 2\delta}$. We can then rescale it
in the region between $D''$ and some large disk $D_r$ (here we do not control
$r$) so that it still has positive divergence and coincides with $Y_0$ outside
$D_r$.
\end{proof}

\begin{proof}[Proof of Proposition~\ref{prop:geomtop_open}]
Since any ball inside $(S^1\times \R^2, \ker(d\phi+\rho^2\, d\theta))$ easily
embeds inside $(\R^3, \xistd)$ we only need to construct an embedding into the
former model. The proposition will follow from the construction of an embedding
of $\T(r)$ into a solid torus satisfying the hypothesis of
Lemma~\ref{lem:ct}. We construct this embedding in several steps.

We first introduce a technical definition. Let $Y$ be either a closed interval
or $S^1$. Let $D^2$ be the unit disk in $\R^2$ and $D_r$ denote the disk of
radius $r.$ A contact structure on a $D^2 \times Y$ is under control if there is
a vector field tangent to $\{0\} \times Y$ and to $\partial D_r \times Y$ for
all $r$ which is transverse to both $\xi$ and the obvious foliation by disks. A
contact structure on a domain $C$ is under control if there is a diffeomorphism
from $C$ to $D^2 \times Y$ sending $\xi$ to a contact structure under
control. The inverse images of the objects involved in the above discussion are
then said to control $\xi$. In particular, in our geometric setup, $\xi$ is
under control on $\T(r)$ and Lemma~\ref{lem:ct} applies to any contact structure
which is under control on a solid torus.

We want to construct thickened disks $\T_t$ and $\T_b$ that can be glued to the
top and bottom, respectively, of $\T(r)$ so that the characteristic foliation on
the top and bottom of $U = \T_t\cup \T(r) \cup \T_b$ is ``standard'' and $\xi$
stays under control on this larger thickened disk.  (Here top and bottom refer to
$\T(r)$ seen with $\zeta$ vertical and oriented from bottom to top. Also
``standard'' means that the singular foliation has a single elliptic singularity
and the rest of the leaves are radial.) Thus we will be able to glue the top and
bottom of $U$ together to obtain a solid torus $S$ with a contact structure
under control into which $\T(r)$ embeds and Lemma \ref{lem:ct} will finish the
proof.

We discuss the construction of $\T_t,$ the construction of $\T_b$ being
analogous.  Let $z_t$ be the top extremity of $\zeta$ and $\D := \D_{z_t}$.  The
previous lemma gives a contact embedding $\varphi$ of a neighborhood of $\D$ in
$\R^2 \times \R$ such that the image of $\D$ is the graph of some function $f$
over some (deformed) disk $\Omega$ in $\R^2$.

Let $K$ be a constant such that $f(\rho,\theta) < K$ for all $(\rho,\theta)\in
\Omega$.  Set $C=\{(\rho,\theta,z)| (\rho,\theta)\in \Omega, f(\rho,\theta)\leq
z\leq K\},$ this will (almost) be the bottom part of $\T_t.$ It is foliated by
the graphs $G_s$ of functions $(1 - s)f + sK$, $s \in [0,1]$ and can also be
seen as the union of the Reeb orbit $\varphi(z_t) \times [f(z_t), K]$ and
vertical annuli over the images $c_r$ of the $\partial\T(r) \cap \D$.  Together
with the vector field $\partial_z$, these objects show that $\xistd$ is under
control on $C$.  We now smooth $\varphi(\T(r) \cap Dom(\varphi)) \cup C$ to get
a thickened disk $C'$ extending $\varphi(\T(r) \cap Dom(\varphi))$ above
$\varphi(D)$ which coincides with $C$ when $z$ is close to $K$ and such that
$\xistd$ is under control on $C'$.

Then we notice that there is a large number $R$ such that $\Omega$ is contained
in the disk of radius $R$ about the origin and there is an isotopy
$\Psi_t:\Omega\to \R^2, t\in [0,1],$ such that $\Psi_0=id_\Omega$ and
$\Psi_1(\Omega)$ is a disk of radius $R$ centered about the origin. Consider the
embedding $\Psi:\Omega\times[0,1]\to \R^3$ defined by
$\Psi(p,t)=(\Psi_t(p),K+K't),$ where $K'$ is a large positive constant to be
determined soon.
For each radius $r$, the characteristic foliation on $\Psi(c_r \times [0,1])$
is given as the kernel of $\Psi|_{c_r\times [0,1]}^* (dz+\rho^2\, d\theta)=K'\,
dt+\beta,$ where $\beta$ is independent of $K'.$ Thus for $K'$ large enough
$\Psi_*\frac{\partial}{\partial t}$ is never tangent to the characteristic
foliation of $\Psi(c_r \times [0,1])$ and $\xistd$ is under control on the image
of $\Psi$. We then choose $\T_t$ to be a smoothed version of $C' \cup \Psi(\Omega
\times [0, 1])$. It can be glued to $\T(r)$ using $\varphi$.
After doing the same thing for the bottom of $\T(r)$ we get $\T_b$ and we can do
the construction of $\T_b$ and $\T_t$ with the same large radius $R$ so that the
top and bottom of $\T_b \cup \T(r) \cup \T_t$ can be glued to get a solid torus
with a contact structure under control.
\end{proof}

\small


\begin{thebibliography}{10}

\bibitem{AlbersHofer09}
Peter Albers and Helmut Hofer.
\newblock On the {W}einstein conjecture in higher dimensions.
\newblock {\em Comment. Math. Helv.}, 84(2):429--436, 2009.

\bibitem{Bennequin83}
Daniel Bennequin.
\newblock Entrelacements et \'equations de {P}faff.
\newblock In {\em Third Schnepfenried geometry conference, Vol. 1
  (Schnepfenried, 1982)}, volume 107 of {\em Ast\'erisque}, pages 87--161. Soc.
  Math. France, Paris, 1983.


\bibitem{Blair02}
David~E. Blair.
\newblock {\em Riemannian geometry of contact and symplectic manifolds}, volume
  203 of {\em Progress in Mathematics}.
\newblock Birkh\"auser Boston Inc., Boston, MA, 2002.


\bibitem{BormanEliashbergMurphy}
Matthew Strom Borman and Yakov Eliashberg and Emmy Murphy.
\newblock {\em Existence and classification of overtwisted contact structures in all dimensions}
\newblock Preprint 2014. 

\bibitem{Chavel06}
Isaac Chavel.
\newblock {\em Riemannian geometry. A modern introduction.}, volume~98 of {\em
  Cambridge Studies in Advanced Mathematics}.
\newblock Cambridge University Press, Cambridge, second edition, 2006.

\bibitem{CieliebakEliashberg}
Kai Cieliebak and Yakov Eliashberg.
\newblock {\em From Stein to Weinstein and back: Symplectic geometry of affine
  complex manifolds}, volume~59.
\newblock American Mathematical Soc., 2012.


\bibitem{Eliashberg92a}
Yakov Eliashberg.
\newblock Contact {$3$}-manifolds twenty years since {J}. {M}artinet's work.
\newblock {\em Ann. Inst. Fourier (Grenoble)}, 42(1-2):165--192, 1992.

\bibitem{EKM-sphere}
John~B. Etnyre and Rafal Komendarczyk and Patrick Massot.
\newblock Tightness in contact metric 3--manifolds.
\newblock \emph{Invent. Math.}  188(3), 621-657, 2012.

\bibitem{Giroux00}
Emmanuel Giroux.
\newblock Structures de contact en dimension trois et bifurcations des
  feuilletages de surfaces.
\newblock {\em Invent. Math.}, 141(3):615--689, 2000.

\bibitem{Gray}
Alfred Gray
\newblock {\em Tubes}, volume
  221 of {\em Progress in Mathematics}.
\newblock Birkh\"auser Verlag, Basel, 2004. 

\bibitem{GreenWu}
Robert~E. Greene and Hung Wu.
\newblock On the subharmonicity and plurisubharmonicity of geodesically
convex functions.
\newblock {\em Indiana Univ. Math. J.}, 22(7):641--653, 1973.


\bibitem{Hofer93}
Helmut~Hofer.
\newblock Pseudoholomorphic curves in symplectizations with applications to the
  {W}einstein conjecture in dimension three.
\newblock {\em Invent. Math.}, 114(3):515--563, 1993.

\bibitem{Karcher89}
Hermann~Karcher.
\newblock Riemannian comparison constructions.
\newblock In {\em Global differential geometry}, volume~27 of {\em MAA Stud.
  Math.}, pages 170--222. Math. Assoc. America, Washington, DC, 1989.



\bibitem{MNW}
Patrick Massot and Klaus Niederkr\"uger and Chris Wendl.
\newblock Weak and strong fillings of higher dimensional contact manifolds.
\newblock {\em Invent. Math.} 192(3):287-373, 2013.


\bibitem{Niederkrueger06}
Klaus Niederkr{\"u}ger.
\newblock The plastikstufe---a generalization of the overtwisted disk to higher
  dimensions.
\newblock {\em Algebr. Geom. Topol.}, 6:2473--2508, 2006.

\bibitem{NiederkrugerPasquotto09}
Klaus Niederkr{\"u}ger and Federica Pasquotto.
\newblock Resolution of symplectic cyclic orbifold singularities.
\newblock {\em J. Symplectic Geom.}, 7(3):337--355, 2009.

\bibitem{Petersen}
Peter Petersen.
\newblock {\em Riemannian geometry}, volume~171 of {\em
  Graduate Texts in Mathematics}.
\newblock Springer, second edition, 2006.

\bibitem{Reinhart}
Bruce Reinhart.
\newblock The second fundamental form of a plane field.
\newblock {\em J. Differential Geom.}, 12(4):619--627 (1978), 1977.

\bibitem{Rukimbira95}
Philippe Rukimbira.
\newblock Chern-{H}amilton's conjecture and {$K$}-contactness.
\newblock {\em Houston J. Math.}, 21(4):709--718, 1995.

\bibitem{Sakai96}
Takashi Sakai.
\newblock {\em Riemannian geometry}, volume 149 of {\em Translations of
  Mathematical Monographs}.
\newblock American Mathematical Society, Providence, RI, 1996.
\newblock Translated from the 1992 Japanese original by the author.

\end{thebibliography}
\end{document}